\documentclass{amsart}
\usepackage[left=1in,right=1in,top=1in,bottom=1in]{geometry}
\usepackage[utf8]{inputenc}
\usepackage{amsfonts}
\usepackage{amsmath}
\usepackage{amssymb}
\usepackage{pgfplots}
\usetikzlibrary{intersections}
\usepackage{enumitem}
\usepackage{tikz}
\usepackage{amsthm}
\usepackage{mathrsfs}
\usepackage{booktabs}
\usepackage{amsmath}
\usepackage{graphicx}
\usepackage{chngcntr}
\usepackage{mathtools}
\usepackage{hyperref}
\usepackage[english]{babel}
\usepackage{biblatex} 
\counterwithin{equation}{section}
\newtheorem{theorem}{Theorem}
\newtheorem{definition}{Definition}[section]
\newtheorem{corollary}{Corollary}[section]

\newtheorem{lemma1}{Lemma}[section]

\newtheorem{remark}{Remark}
\DeclareMathOperator{\diam}{diam}

\title[Hausdorff dimension of the set of eventually always hitting points]{\textbf{Hausdorff dimension of the set of eventually always hitting points on a self-conformal set}}

\author{Xintian Zhang}
\address{Institute of mathematics, University of Bristol}
\email{AY19811@bristol.ac.uk}
\date{November 20, 2023}
\thanks{%
This research was supported by the Engineering and Physical Sciences Research Council (EPSRC) under grant number EP/T517872/1.  
}

\begin{document}
\maketitle
\newcommand{\floor}[1]{\lfloor #1\rfloor}
\newcommand{\bfloor}[1]{\big\lfloor #1\big\rfloor}
\newcommand{\bbfloor}[1]{\bigg\lfloor #1\bigg\rfloor}
\newcommand{\dimH}[1]{\dim_\mathscr{H}(#1)}
\newcommand{\dimL}[1]{\dim_\mathscr{L}(#1)}
\newcommand{\vl}{v_s(\E,\T)}
\newcommand{\vh}{v_e(\E,\T)}
\newcommand{\I}{\mathbf{i}}
\newcommand{\J}{\mathbf{j}}
\newcommand{\C}{\mathbf{s}}
\newcommand{\R}{\mathbf{r}}
\newcommand{\T}{\mathbf{t}}
\newcommand{\E}{\mathbf{e}}
\newcommand{\NA}{\Sigma^{\mathbb{N}}}

\begin{abstract}
    Recurrence problems are fundamental in dynamics, and for example, sizes of the set of points recurring infinitely often to a target have been studied extensively in many contexts. For example, the problem of finding the dimension for shrinking target set in an iterated function system is an active research area. In the current work, we consider a set with a finer recurrence quality, the eventually always hitting set. In a sense, the points in the intersection of an eventually always hitting set and a shrinking target set not only return infinitely often but also at a bounded rate. We study this set in the context of self-conformal iterated function systems, and compute upper and lower bounds for its Hausdorff dimension. Additionally, as an intermediate theorem, we obtain a Hausdorff dimension result for the intersection of eventually always hitting and shrinking target sets.
\end{abstract}

\section{Introduction}
The study of subsets in dynamical systems with particular recurrence properties is an active research area with a long history. The classical Poincaré recurrence theorem \cite{po} served as a starting point to the research into properties of sets of infinitely often recurring points, such as the limsup sets known as shrinking target sets. The term “shrinking target” was coined by Hill and Velani in \cite{Hill1995}, although research on this topic predates that publication. Meanwhile, the liminf counterpart to shrinking targets; that is, the sets of eventually always hitting points have not attracted attention until recently. 

We will refer to the set of eventually always hitting points in a set as the eventually always hitting set for simplicity. Both eventually always hitting sets and shrinking target sets are important and natural concepts in the study of dynamical systems. Let us define these notions. For a dynamical system $(X, T)$, let $(B_i)_i$ be a sequence of subsets of $X$. We define the \textit{shrinking target set} with respect to $(B_n)_n$ as follows:
\[R_s((B_n)_n)=\{x\in X|\forall N\in \mathbb{N},\exists n \geq N \text{ such that } T^n(x)\in B_n \} .\]
We also define the \textit{eventually always hitting set} as:
\[R_e((B_n)_n)=\{x\in X | \exists N\in\mathbb{N}, \forall n\geq N  \text{ we have } T^{m}(x)\in B_{n}   \text{ for some } m\leq n   \}.\]
At first glance, this set may not appear to be a natural counterpart to the shrinking target set. However, the structures of these two sets are strongly connected as discussed in detail in  Chapter 2.
In this paper, we limit our attention to these sets in the context of fractal geometry. Before introducing the results in the literature, we first define some notions.\\
Let $O\subset \mathbb{R}^d$ for some $d\in \mathbb{N}$. Let $f_i:O\to \mathbb{R}^d$ for all $i \in \{1,\cdots,S\}$ where $S\in \mathbb{N}$ and $S\geq  2$. If there exist a closed sets $M\subset O$ such that $f_i(M)\subset M$ for all $i\in \{1,\cdots,S\}$, we call $F=\{f_1,...,f_S\}$ as an \textit{iterated function system (IFS)}. According to a well-known result by Hutchinson \cite{Hut}, there exists a unique non-empty compact set $\Lambda \subset M$ such that $\Lambda = \bigcup_{i}^{N} f_i(\Lambda)$. This unique set is usually referred to the \textit{attractor} of the IFS. We have three classical categories of the attractors with respect to the IFS that generate them. If $f_i$'s are all similarities, we call the associated attractor a \textit{self-similar set}. In particular, if the contraction ratios are all equal in $f_i$ for $i\in \{1,\cdots,S\}$, we say the self-similar is homogeneous. If $f_i$'s are all affine maps, we call the associated attractor a \textit{self-affine set}. The third class that is often studied are the self-conformal sets. There are slight differences between the definitions of conformal IFS in the literature. Here, we adopt the definition provided in \cite{Sascha}. Fix an open set $U\subset \mathbb{R}^n$. A $C^1$-mapping $f:U\to \mathbb{R}^d$ is \textit{conformal} if the differential $f'(x):\mathbb{R}^d\to\mathbb{R}^d$ is a similarity, i.e. satisfies $|f'(x)y|=|f'(x)||y|\neq 0$ for all $x\in U$ and $y\in \mathbb{R}^d \backslash \{0\}$ and, as a function of $x$, is Hölder continuous, i.e. there exist $\alpha,c>0$ such that $|f'(x)-f'(y)|\leq c|x-y|^{\alpha}$ for all $x,y\in U$. If $f_i$ is an injective conformal mapping restricted on a bounded open convex set $U$ such that ${f_i(\overline{U})}\subset U$ where $\overline{U}$ is the closure of $U$ and $\|f'_i\|:=\sup_{x\in U}|f'_i(x)|<1 $ for all $i\in \{1,\cdots,S\}$, we say the attractor of this IFS is a \textit{self-conformal set}.

For the shrinking target problem, Shen and Wang \cite{Shen2013} calculated the Hausdorff dimension of the shrinking target set for homogeneous self-similar IFS. Hill and Velani \cite{URBANSKI2005219}, Allen and Bárány \cite{Demi} calculated the Hausdorff dimension and Hausdorff measure of the shrinking target set for self-conformal IFS respectively. With some algebraic conditions, Koivusalo and Ramírez \cite{koivusalo_ramírez_2018} calculated the Hausdorff dimension of the shrinking target set for self-affine IFS for generic targets; Bárány and Troscheit \cite{bar} refine these algebraic conditions. Bárány and Rams\cite{B_r_ny_2018} calculated the Hausdorff dimension of the shrinking target set for a explicit type of self-affine IFS, namely Bedford-McMullen carpet.

On the other hand, in the context of eventually always hitting sets much less is known. Some dynamic properties and Borel-Cantelli type measure results are discussed in \cite{GP} by Ganotaki and Persson, in \cite{Kirsebom_2020} by Kirsebom, Kunde and Persson, and in \cite{holland2021dichotomy} by Holland, Kirsebom, Kunde and Persson. Their works also focus more on the intervals instead of fractal sets. 

To the best of our knowledge, the only Hausdorff dimension results for the eventually always hitting set are provided by Bugeaud and Liao in \cite{bugeaud} and by Zheng and Wu in \cite{Zheng_2020}, where they calculated the Hausdorff dimension for the eventually always hitting set in the structure of $\beta$-expansion interval maps—a subclass of the homogeneous self-similar IFS.

In this paper, we do not only limit ourselves to the interval maps. Instead, we generalize the Hausdorff dimension result to self-conformal IFS. We calculate the upper and lower bound of the Hausdorff dimension of eventually always hitting set for self-conformal IFS. We also provide an explicit condition for when the upper and lower bound coincide, and this condition is satisfied in many common scenarios. We provide a more precise outline of our results in the next section. The comparison to Bugeaud and Liao's result will be explained in more detail in the remarks in Chapter 5 after the ideas of the proof are clear. 

\subsection{Background and notations}
\begin{flushleft}
     We say that an IFS $F$ satisfies the \textit{open set condition} if we can find an open set $O$ such that $f_i(O) \subset O$ and $f_i(O) \cap f_j(O) = \emptyset$ for all $i, j \in {1,...,S}$, where $i \neq j$. For an IFS $F$ that satisfies the open set condition, recall that $\Lambda$ is the attractor of F. We define a map $\pi$ from $\Sigma^{\mathbb{N}} := \{1,...,S\}^{\mathbb{N}}$ to $\Lambda$ by $\pi(\mathbf{i} = a_1a_2...) := \lim_{n\to \infty} f_{a_1} \circ f_{a_2} \circ \cdots \circ f_{a_n} (\Lambda)$, where $\I \in \{1,...,S\}^{\mathbb{N}}$.
\end{flushleft}

\begin{flushleft}
Let $k\in \mathbb{N}$, for $\I:=i_1i_2\cdots\in \Sigma^{\mathbb{N}}$ and $m,n \in \mathbb{N}$, we define  
\begin{itemize}[itemsep= 6pt]
 \item $\I|_n^m:=
    \begin{cases}
       a_n a_{n+1}\cdots a_{m}, &\text{ for } m\geq n\\
       \emptyset, &\text{ for } m<n
    \end{cases}$
\item $[\I|^m_n]:=\{\I|^m_n\J:\J\in \Sigma^\mathbb{N}\}$.
\item $\sigma:\Sigma^\mathbb{N} \to \Sigma^\mathbb{N}$ is the left shift map such that for $\I=i_1i_2\cdots\in \NA$, $\sigma(\I):=i_2 i_3\cdots $.
\item $f_{\I|^m_n}:=
    \begin{cases}
       f_{i_n}\circ \cdots \circ f_{i_m}, &\text{ for } m\geq n\\
       identity, &\text{ for } m<n.
    \end{cases}$
    \item $\Sigma_*$ is the collection of $\I|_i^j$ for $i,j\in\mathbb{N}$.
    \item $\|f'_{\max}\|:= \max_{(i,x)\in\{1,\cdots,S\}\times \Lambda} \{|f'_i(x)|\} $.
    \item $\|f'_{\min}\|:= \min_{(i,x)\in\{1,\cdots,S\}\times \Lambda} \{|f'_i(x)|\} $.
\end{itemize}
Notice that since $\Lambda$ is compact $\|f'_{\max}\|$ and $\|f'_{\min}\|$ are well defined, and we have $0< \|f'_{\min}\|\leq \|f'_{\max}\|< 1$.
\end{flushleft}
 
The following lemmas are classic, some relevant results can be found in \cite{Sascha,Falconer,K_enm_ki_2008}, but we recall them for the convenience of the reader. Notice that for a set $A\subset \mathbb{R}^d$, $\diam(A):= \sup_{x,y\in A}\{ d(x,y)\}$ is the diameter of the set $A$.
\begin{lemma1} \text{(Lemma 6.1 in \cite{Sascha})}
For a conformal iterated function system $\{f_i\}_1^S$ defined on $U$, there exists a bounded open
convex set $V\subset \mathbb{R}^d$ such that $f_i(V)\subset V\subset \Bar{V} \subset U$ for all $i\in \{1,\cdots,S\}$. Furthermore, if $\Lambda\subset V$ is the associated self-conformal invariant set containing at least two points, then there exists constant $K\geq 1$ such that
\begin{align}\label{6.1}
    K^{-1}\|f'_\I\|d(x,y)\leq d(f_\I(x),f_\I(y))\leq \|f'_\I\|d(x,y)
\end{align}
for all $x,y\in V $ and $\I\in \Sigma_*$,
\begin{align}\label{6.2}
    \frac{1}{\diam(\Lambda)}\diam(f_\I(\Lambda))\leq \|f'_\I\|\leq \frac{K}{\diam(\Lambda)}\diam(f_\I(\Lambda))
\end{align}
for all $\I\in \Sigma_*$, and
\begin{align}\label{6.3}
    K^{-2}\|f'_\I\|\|f'_\J\|\leq\|f'_{\I\J}\|\leq \|f'_\I\|\|f'_\J \|
\end{align}
for all $\I,\J\in \Sigma_*$.\\
Furthermore, if $\{f_i\}_1^N$ satisfies the open set condition, we can find an open set $O\subset V$ and $K' \geq 1$ such that
\begin{align}\label{6.4}
    \frac{1}{\diam(O)}\diam(f_\I(O))\leq \|f'_\I\|\leq \frac{K'}{\diam(O)}\diam(f_\I(O)).
\end{align}
for all $I\in\Sigma_*$.
\end{lemma1}

Our primary objective in this paper is to determine the Hausdorff dimension of the eventually always hitting set $R_e((B_n)_n)$ on the symbolic space projected on self-conformal sets that satisfy the open set condition. Notice that here $B_i$'s are just $[\I|_1^{a_i}]$ where $(a_n)_n$ is just a sequence of natural numbers. To provide a precise statement of the theorems and results, we introduce the following notations.

Let $\T, \E\in \Sigma^{\mathbb{N}}$, then  $\pi(\T)$ and $\pi(\E)$ are points in the self-conformal set. Let $(a_n)_n$ be a sequence of natural numbers. We introduce the notations $R_e((a_n)_n,\T)$ and $R_s((a_n)_n,\T)$ to represent the eventually always hitting set and shrinking target set on symbolic space, respectively. Specifically, we define them as follows:
\begin{align*}
R_e((a_n)_n,\T) := R_e(([\T|_1^{a_n}])_n) \quad \text{and} \quad R_s((a_n)_n,\T) := R_s(([\T|_1^{a_n}])_n),
\end{align*}
for convenience. In particular, let $v\in [0,\infty)$, we denote
\begin{align*}
R_e(v,\T) := R_e((\lfloor v n \rfloor)_n,\E) \quad \text{and} \quad R_s(v,\T) := R_s((\lfloor v n \rfloor)_n,\E),
\end{align*}
where $\lfloor \cdot \rfloor$ denotes the floor function.

Shifting our perspective to individual points, for $\T,\E\in \Sigma^{\mathbb{N}}$, we define $v_e(\E,\T)$ as the maximum eventually always hitting rate of $\E$ with respect to $\T$, and $v_s(\E,\T)$ as the maximum shrinking target rate of $\E$ with respect to $\T$. The precise definitions are as follows:
\begin{align*}
v_e(\E,\T) := \sup\{v \in \mathbb{R} : \E \in R_e(v,\T)\} \quad \text{and} \quad v_s(\E,\T) := \sup\{v \in \mathbb{R} : \E \in R_s(v,\T)\}.
\end{align*}
It can be noticed easily that 
\begin{align*}
  R_e(v,\T)=\{\E\in \NA: v_e(\E,\T)\geq v\} \text{ and }
  R_s(v,\T)=\{\E\in \NA: v_s(\E,\T)\geq v\}.
\end{align*}
Here we define a subset of $\Sigma^{\mathbb{N}}$ containing our target $\T$ that we considered in the current work.
\begin{definition}
Let's define $\T(n,j)$ to be $\T$ after replacing the $n$-th digit by $j\in \Sigma$. Then we define a subset $G$ of $\Sigma^{\mathbb{N}}$ as
\[G:=\{\T\in \NA: \exists N\in\mathbb{N} \text{ s.t. } \forall n>N \text{ we have } \forall m<n \text{ and } j\neq \T|_n^n \text{ that } \T(n,j)|_{m}^{n}\neq \T|_1^{n-m+1}\}.\]
\end{definition}
\begin{remark}
We restrict our target $\T$ to $G$ to exclude a variation of self recurrence behaviour of the target $\T$. That behaviour could cause technical issue which our current method cannot handle. However, $\T\in G$ still generalizes the results of in current literature. If there are more than 2 maps in the IFS, it is not hard to obtain that $G$ is at least uncountable.  
\end{remark}

In order to state our main theorems, we introduce the following technical notations.
\begin{definition} \label{def}
For any $\T\in G$, $\theta\in(1,\infty)$ ,$v\in (0,1)$, $M\geq \frac{1}{\theta v}$ and $p:=\max \{p\in \mathbb{N}: v M/\theta^p\geq 1\}$, we define: 
    \begin{itemize}[itemsep= 3pt]
    \item $\Sigma^*:= \{\I^*:\I^*\in \Sigma^{N} \text{ for some }N\in\NA\}$.
    \item $G(M,\T,\theta,v):=\T|_1^{\floor{\frac{v M}{\theta^p}}}\T|_1^{\floor{\frac{v M}{\theta^{p-1}}}}\cdots\T|_1^{\floor{\frac{v M}{\theta}}}\T|_1^{\floor{v M}} \T|_1^{\floor{\theta v M}}$
    \item $\Omega^+{(\T,\theta,v)}:=\limsup_{M\to \infty} \frac{\log \|f'_{G(M,\T,\theta,v)}\|}{M}.$
    \item $\Omega^-{(\T,\theta,v)}:=\liminf_{M\to \infty} \frac{\log \|f'_{G(M,\T,\theta,v)}\|}{M}.$
    \item $O(\delta)$ is a function  of $\delta$ satisfying that for some $c\in\mathbb{R}$, we have $ O(\delta) \leq c \delta $.
    \item $\Delta(\T,\theta,v):= \{\E\in\NA:v_e(\E,\T)\geq v\} \cap \{\E\in\NA:v_s(\E,\T)\geq \theta v\}$.
    \item $\Upsilon(\T,\theta,v,\delta):= \{\E\in\NA:v_s(\E,\T)\geq v\} \cap \{\E\in\NA:(\theta-\delta) v  \leq  v_s(\E,\T)\leq(\theta+\delta) v\}$.
\end{itemize}
\end{definition}
The pressure $P:[0,+\infty]\to \mathbb{R}$ is defined by 
\[P(s)=\lim_{n\to\infty}\frac{1}{n}\log\sum_{\I\in\Sigma^n}\|f'_\I\|^s.\]
Since we are considering self-conformal IFS, $P$ is a well-defined, continuous, differentiable, convex and strictly decreasing function, check Chapter 5 of \cite{Falconer} and Section 2.5 of \cite{NA2018}. Thus, it is not hard to obtain that the following quantities are well defined.
\begin{definition}\label{def-}
    We define $s^+(\T,v,\theta)$ to be the unique solution of the equation
\[P(s^+(\T,v,\theta))= -s^+(\T,v,\theta)\frac{\theta-1}{\theta-\theta v-1}\Omega^+(\T,v,\theta)\]
and $\hat{s}^+(\T,v):= \sup_{\theta\in (\frac{1}{1-v},+\infty)} s^+(\T,v,\theta)  $. In comparison, we define $s^-(\T,v,\theta)$ to be the unique solution of the equation
\[P(s^-(\T,v,\theta))= -s^-(\T,v,\theta)\frac{\theta-1}{\theta-\theta v-1}\Omega^-(\T,v,\theta)\]
and $\hat{s}^-(\T,v):= \sup_{\theta\in (\frac{1}{1-v},+\infty)} s^-(\T,v,\theta)  $.
Finally, we denote two functions $\omega^+: \Lambda\times [0,1]\to \mathbb{R}$ by 
\[
    \omega^+(\T,v):= 
\begin{cases}
    \dimH{\Lambda},& \text{if } v=0,\\
    \min\{\hat{s}^+(\T,v), \dimH{\Lambda} \}           & \text{if } v\in(0,1),\\
    0 & \text{if } v=1.
\end{cases}
\] and $\omega^-: \Lambda\times [0,1]\to \mathbb{R}$ by
\[
    \omega^-(\T,v):= 
\begin{cases}
    \dimH{\Lambda},& \text{if } v=0,\\
    \min\{\hat{s}^-(\T,v), \dimH{\Lambda} \}           & \text{if } v\in(0,1),\\
    0 & \text{if } v=1.
\end{cases}\] 
\end{definition}
Notice that $\omega^+(\T,v)$ and $\omega^-(\T,v)$ are continuous, and the proof will be given later.

\subsection{Main results}
\begin{flushleft}
   Our results are as follows: Theorem 1 stands as our main result, while Theorem 2 serves as a technical theorem that can be utilized in the proof of Theorem 1. 
\end{flushleft}

\begin{theorem} \label{theorem1}
\begin{flushleft}
Let $F$ be a conformal IFS that satisfies the open set condition and $\Lambda$ be the self-conformal set generated by $F$. For a sequence of natural numbers $(a_n)_n$ that satisfies $\lim_{n\to\infty}\frac{a_n}{n}=v$, and $\T\in G$, we have
\newline
\textit{Case }1: If $v\in [0,1],$
\[\omega^-(\T,v)\leq\dimH{\pi(R_e((a_n)_n,\T))}\leq \omega^+(\T,v)\]
\textit{Case }2: If $v\in (1,+\infty)$,
\[R_e((a_n)_n,\T)=\emptyset.\]
\textit{Case }3: If $v=+\infty$,
\[R_e((a_n)_n,\T) \text{ is countable.}\]
In particular, if for all $\theta\in (\frac{1}{1-v},\infty)$
\begin{align}\label{condition}
  \Omega^+{(\T,\theta,v)}=\Omega^-{(\T,\theta,v)},  
\end{align}
then we have $\omega^+(\T,v)=\omega^-(\T,v)$ and
\[\dimH{\pi(R_e((a_n)_n,\T))}=\omega^+(\T,v)=\omega^-(\T,v).\]
\end{flushleft}
\end{theorem}

\begin{theorem}\label{theorem2}

Let $F$ be a conformal IFS that satisfies the open set condition and $\Lambda$ be the self-conformal set generated by $F$. For $v\in (0,1)$ and $\T\in G$, the followings holds:\newline
(1) If $\theta< \frac{1}{1-v}$, then
\[\{\E\in\NA:v_e(\E,\T)\geq v\} \cap \{\E\in \NA:v_s(\E,\T)=\theta v\}=\emptyset\]
(2) If $\theta\geq \frac{1}{1-v} $, we have for $\delta>0$ small enough,
\begin{enumerate}
    \item $\dimH{\pi(\Upsilon(\T,\theta,v,\delta))}\leq\min\{s^+(\T,v,\theta)+O(\delta),\dimH{\Lambda}\}$;
    \item $\dimH{\pi(\Delta(\T,\theta,v))}\geq \min\{s^-(\T,v,\theta),\dimH{\Lambda}\}$.
\end{enumerate}
\end{theorem}

The proofs of our results are divided into three sections. Section 2 focuses on the symbolic space and discusses the range of $v_e$ and $v_s$, as well as their relationship, and provides a proof for Theorem 2 part (1). Section 3 serves as the main section of proofs, providing a proof for Theorem 2 part (2).
Section 4 combines proof for Theorem 1 as a consequence of Theorem 2. Our strategy involves using Theorem 2 to establish Theorem 1 with $(a_n)_n=(\lfloor nv \rfloor)_n$ for a suitable value of $v$. We then extend this result to Theorem 1 by employing Lemmas which will be introduced later. 

Section 5 introduces some explicit examples of when the condition \ref{condition} is satisfied. Some comments on current methods and possible further research directions are also included in Section 5.

\section{The influence of shrinking speed}
In this section, we introduce some lemmas that provide basic properties of the symbolic structures of these two dynamically defined sets: shrinking target sets and eventually always hitting sets. As a result, we provide a proof of Theorem 2 part (1).
\subsection[Possible values of $v_e$ and $v_s$]{Possible values of $\texorpdfstring{v_e}{ve}$ and $\texorpdfstring{v_s}{vs}$}

\begin{flushleft}
    We survey of all possible values of $v_e(\E,\T)$ and $v_s(\E,\T)$, and the meaning of these symbolic representations. For $\T\in G$, define $\Lambda_{\T}:=\{\J\T|\J\in \Sigma^*\}$, the following lemmas holds.
\end{flushleft}

\begin{lemma1}\label{2.1-}
Let $\T\in G$ and $\E \in \NA$. We have $v_e(\E,\T)\in [0,1]\cup \{+\infty\} $ and $v_e(\E,\T)=+\infty$ if and only if $\E\in\Lambda_{\T}$.
\begin{proof}
Let us first note that by definition, $v_e(\E,\T)\geq 0$. It is not hard to check that $\E\in\Lambda_{\T}$ implies $v_e(\E,\T)=+\infty$. Then, assuming $\vh=\alpha>1$, we aim to show that $\alpha=+\infty$. By definition of $v_e(\E,\T)$, for $\epsilon\in (0,\frac{\alpha-1}{2})$, there exists $N\in \mathbb{N}$ such that for all $n>N$, exists $m\leq n$ such that
\begin{align}\label{Lemma2.1.1}
    \E|_{m+1}^{m+\floor{(\alpha-\epsilon) n}}=\T|_1^{\floor{(\alpha-\epsilon)n}}.
\end{align}
We denote by $m_n$ the smallest $m$ satisfying (\ref{Lemma2.1.1}). If $\E|_{m_n+1}^{m_n+t}=\T|_1^{t}$ for all $t\in \mathbb{N}$, we can conclude that $\E\in \Lambda_{\T}$ which implies $v_e(\E,\T)=\infty$.
Otherwise, there exist exist $k\in\mathbb{N}$, such that, $\E|_{m_n+1}^{m_n+k}\neq\T|_1^{k}$. We denote by $k_{n}$ the smallest number such that $\E|_{m_n+1}^{m_n+k_{n}}\neq\T|_{1}^{k_{n}}$. Again, by definition of $\vh=\alpha$, there exists some $q\leq m_n+k_{n}$ such that $\E|_{q+1}^{q+\floor{(\alpha-\epsilon) (m_n+k_{n})}}=\T|_1^{\floor{(\alpha-\epsilon) (m_n+k_{n})}}$. Notice that $q>m_n$ and $\floor{(\alpha-\epsilon)(m_n+k_{n})}>k_{n}$ implies $\E|_{q}^{q+k_n}=\T|_{1}^{k_n}$ which contradicts to the definition of $G$. Therefore, we conclude that either $v_e(\E,\T)=+\infty$ and $\E\in \Lambda_{\T}$, or $v_e(\E,\T)\leq 1$.
\end{proof}
\end{lemma1}

Notice that Lemma \ref{2.1-} includes Case 3 of Theorem \ref{theorem1}.

\begin{lemma1}
Let $\T,\E\in \NA$. If $\E \in \Lambda_{\T}$, we have 
\[v_s(\T,\T)= \vl\leq \vh=+\infty.\]
Otherwise,
\[0\leq \vh\leq\vl\leq+\infty.\]
\end{lemma1}
\begin{proof}

If $\E\in \Lambda_{\T}$, by Lemma \ref{2.1-}, $v_e(\E,\T)=+\infty$. Thus, 
\[\vl\leq \vh=+\infty.\]
We prove the case that $v_s(\T,\T)= \vl$.
If there exists infinitely many $j\in \mathbb{N}$ and $t_j\in \mathbb{N}$  such that
\[\E|_{j+1}^{j+t_j}=\T|_1^{t_j},\] without loss of generality, we claim that $(j_n)_n$ is an increasing sequence that contains all such $t$ and for every $j_n$, $t_{j_n}$ is as large as possible. By definition, $v_s(\E,\T)=\limsup_{n\to\infty}\frac{t_{j_n}}{j_n}$. As $\I_{\E}\in \Lambda_{\T}$, there exist $M_{\E,\T}$ such that $\sigma^{M_{\E,\T}}(\E)=\T$. Therefore, for all $j_n\geq M_{\E,\T}$, 
\[\T|_{j_n-M_{\E,\T}+1}^{j_n-M_{\E,\T}+t_{j_n}}=\T|_1^{t_{j_n}}.\]
Thus, $v_s(\E,\T)=\limsup_{n\to\infty}\frac{t_{j_n}-M_{\E,\T}}{j_n}=\limsup_{n\to\infty}\frac{t_{j_n}}{j_n}=v_s(\E,\T)$. If there is no such sequence $(j_n)_n$, it not hard to obtain that $v_s(\E,\T)=v_s(\E,\T)=0$. Therefore, 
\[v_s(\T,\T)= \vl\leq \vh=+\infty.\]
If $\E  \notin \Lambda_{\T}$,
the case when $v_e(\E,\T)=0$ or $v_s(\E,\T)$ is trivial. Thus, from Lemma \ref{Lemma2.1.1} we can say that for some $\alpha\in (0,1)$, $0<\vh=\alpha\leq 1$. For any fixed $\epsilon>0$, we have $N\in\mathbb{N_\epsilon}$ such that for $n\geq N_\epsilon$, there exists $m_n\leq n$ with 
\begin{align}\label{Lemma2.2.1}
    \E|_{m_n+1}^{m_n+\floor{(\alpha-\epsilon) n}}=\T|_1^{\floor{(\alpha-\epsilon)n}}.
\end{align}
By Lemma \ref{2.1-}, we have $\E\notin \Lambda_{\T}$ which implies that there exist $k_n\geq m$ such that $\E|_{m+1}^{m+k_n}\neq \T|_1^{k_n}$.
Therefore, by taking a sequence $(n_l)_l$ such that $n_l=k_{n_{l-1}}+1$, the following holds
\[  \E|_{m_{n_l}+1}^{m_{n_l}+\floor{(\alpha-\epsilon) m_{n_l}}}=\T|_1^{\floor{(\alpha-\epsilon)m_{n_l}}}.\]
Thus, by definition of $v_s(\E,\T)$,
\[v_s(\E,\T)\geq \frac{(\alpha-\epsilon)m_{n_l}}{m_{n_l}}=\alpha-\epsilon.\]
By taking $\epsilon\to 0$, we have $v_s(\E,\T)\geq \alpha = v_e(\E,\T)$. Thus,
$ 0\leq \vh\leq\vl\leq+\infty.$

\end{proof}

\subsection{Relation between \texorpdfstring{$v_e$}{ve} and \texorpdfstring{$v_s$}{vs}}
\begin{flushleft}
   The followings show that if $v_s(\E,\T)< \infty$, we have a more precise relation between $v_e(\E,\T)$ and $v_s(\E,\T)$. We can notice that both $v_e$ and $v_s$ give some information on the asymptotic behavior of the symbolic structure. Here, we present the relation between $v_s(\E,\T)$ and $v_e(\E,\T)$. The proof of the following Lemma follows section 2.1 in \cite{bugeaud} very closely, but we include it for the reader's convenience.
\end{flushleft}

\begin{lemma1} \label{lemma:relation}
For $\E,\T\in \Lambda$ and $v_s(\E,\T)<\infty$, we have $v_s(\E,\T)\geq \frac{v_e(\E,\T)}{1-v_e(\E,\T)}$.
\end{lemma1}

\begin{proof}

We construct two non-decreasing sequences $(m'_k)_k$ and $(n'_k)_k$ satisfying: 
\begin{itemize}
    \item For all $k\in \mathbb{N}, n'_k\leq m'_k$,
    \item Every $ j\in \mathbb{N}$ satisfying $ \E|_{j}^{j} = \T|_1^1$, also satisfies $j= n'_k $ for some $k$,
    \item $\E|_{n'_k+1}^{m'_k-1}=\T|_{1}^{m'_k-n'_k-1}, 
 \E|_{n'_k}^{n'_k}\neq \T|^1_1, 
 \E|_{m'_k}^{m'_k}\neq \T|_{m'_k-n'_k}^{m'_k-n'_k} $.
\end{itemize}
Since $v_e(\E,\T)>0$, we have 
\[\limsup_{k\to\infty}(m'_k-n'_k)= +\infty\]
Therefore, we can inductively define subsequences $(m_k)_k$ and $(n_k)_k$ out of $(m{'}_k)_k$ and $(n'_k)$ that satisfy $(m_k-n_k)_k$ being increasing sequence. To be more precise, we define $m_1=m'_1$ and $n_1=n_1'$. Assume for $k\geq 1$, we have defined $m_k=m'_{j_k}$ and $n_k=n'_{j_k}$, then we define $j_{k+1}:=\min\{j>j_k: m'_j-n'_j > m_k-n_k\}$. Therefore, by defining $n_{k+1}=n'_{j_{k+1}}$ and $m_{k+1}=m'_{j_{k+1}}$, we finished the inductive construction. Notice that $\E|_{n_{k'}}^{m_{k'}}$ is the sequence of maximized subwords coinciding with the target word $\T|_1^{m_{k'}-n_{k'}-1}$. On the other hand, we take subsequences $(k)_k$ from $(k')_{k'}$ to guarantee that $m_k-n_k$ increases with respect to $k$. Thus, we have by the definition of $G$
\begin{align}
  v_s(\E,\T)&=\limsup_{k\to \infty}\frac{m_k-n_k}{n_k}=\limsup \frac{m_k}{n_k}-1 \label{2.1}\\
  v_e(\E,\T)&=\liminf_{k\to+\infty}\frac{m_k-n_k}{n_{k+1}}\leq\liminf_{k\to \label{2.2} +\infty}\frac{m_k-n_k}{m_k}=1-\limsup_{k\to +\infty}\frac{n_k}{m_k}.
\end{align}
Therefore, by \[(\limsup_{k\to+\infty}\frac{n_k}{m_k})\cdot(\limsup_{k\to+\infty}\frac{m  _k}{n_k})\geq  1\] 
and equation (\ref{2.1}) and (\ref{2.2}) we have 
\begin{align}
  v_e(\E,\T)\leq 1-\frac{1}{1+v_s(\E,\T)}=\frac{v_s(\E,\T)}{1+v_s(\E,\T)}.\label{2.3}
\end{align}
We assumed that $v_s(\E,\T)$ is finite, and that implies $v_e(\E,\T)<1$. Hence, (\ref{2.3}) implies
\[v_s(\E,\T)\geq \frac{v_e(\E,\T)}{1-v_e(\E,\T)}.\qedhere \]
\end{proof}

The sequences $(n_k)_k$ and $(m_k)_k$ defined in Lemma \ref{lemma:relation} will be used throughout this paper. It follows from the definition that the sequences $(m_k)_k$ and $(n_k)_k$ are unique if $\E,\T$ are given. We end this section with the proof of Theorem 2 part (1): 
\begin{proof}[Proof of Theorem 2 part (1)]
   Assume there exist $\E_0\in \NA$ such that $\pi(\E_0)\in \{\E\in\NA:v_e(\E,\T)\geq v\} \cap \{\E\in\NA:v_s(\E,\T)=\theta v\}$, for some $\theta< \frac{1}{1-v}$. Lemma \ref{lemma:relation} implies that
\[v\leq v_e(\E_0,\T) \leq \frac{\theta v}{1+\theta v}.\]
Rearranging it results in $\theta\geq \frac{1}{1-v}$. Hence $\{\E\in\NA:v_e(\E,\T)\geq v\} \cap \{\E\in\NA:v_s(\E,\T)=\theta v\}$ is non-empty only when $\theta<\frac{1}{1-v}$. 
\end{proof}
\begin{remark}
     If we change our target balls to annulus in the eventually always hitting set, we can notice from equation (\ref{2.2}) and equation (\ref{2.1}) that after the $n$-th return, the next return needs to be soon enough which depends on how close the current return to the target point. This is the reason why we introduce the eventually always hitting property as a frequency restriction to shrinking target set. Even through in the original form they may seem irrelevant. 
\end{remark}

\section{Hausdorff dimension of the intersections: Proof of Theorem 2}
We provide the proof of Theorem 2 part (2)1 and (2)2 in this section.
\subsection{Ideal covering: upper bound}

\begin{flushleft}
 The construction of a cover of $\pi(\Upsilon(\T,\theta,v,\delta))$ is introduced in this subsection. The Hausdorff dimension is then being estimated and used as an upper bound of $R_e(\T,v)$ for proper $\theta$.    
\end{flushleft}

The following lemma provides an answer to the question: if $\E\in \Upsilon(\T,\theta,v,\delta)$, what can we say about the symbolic structure of $\E$?
\begin{lemma1}
Let $\T\in G$, $v\in(0,1)$, $\theta\geq \frac{1}{1-v}$, $\delta>0$ small and $\E\in \Upsilon(\T,\theta,v,\delta)$. For $\epsilon>0$, there exists a $\hat{k}\in\mathbb{N}$ and an increasing sequence $(k_i)_i$ such that for all $k_i\geq \hat{k}$ in that sequence, we have
\begin{align}\label{eq:13.1}
    \sum_{i=1}^{k_i} m_j-n_j-2 \geq n_{k_i}\left(\frac{(\theta+\delta) v }{\theta+\delta-1}-\epsilon'\right),
\end{align}
where $\epsilon'= \epsilon\left(\frac{(\theta+\delta)^2+1}{(\theta+\delta-1)^2}+1\right).$

\begin{proof}
 Recalling from definition \ref{def}, we have $v_e(\E,\T)\geq v$ and $ (\theta-\delta)v \leq v_s(\E,\T)\leq (\theta+\delta)v$. Recalling the subsequences $m_k$ and $n_k$ from Lemma \ref{lemma:relation}, we assume the axiom of choice and take two subsequences of $(m_k)_k$ and $(n_k)_k$ along which the supremum of (\ref{2.1}) is obtained and define them as $(m_{k_i})_i$ and $(n_{k_i})_i$ respectively. Without loss of generality, let $0<\epsilon<\frac{\theta v}{2}$, $0<\delta<\theta$, by (\ref{2.1}) and (\ref{2.2}), there exists $k(\epsilon,\delta,\E,\T)$ depending on $\epsilon, \delta$ and $\T$, such that for all $k>k(\epsilon,\delta,\E,\T)$ and $k_i> k(\epsilon,\delta,\E,\T)$ in the sequence $(k_i)_i$, the following inequalities hold,
\begin{align}\label{eq:9}
    m_{k}-n_{k}\leq ((\theta+\delta)v+\epsilon)n_{k},
\end{align}
\begin{align}\label{eq:10}
    m_k-n_k\geq (v-\epsilon)n_{k+1},
\end{align}
\begin{align}\label{eq:9+}
    m_{k_i}-n_{k_i}\geq ((\theta-\delta)v-\epsilon)n_{k_i}.
\end{align}
Combining (\ref{eq:9}) and (\ref{eq:10}) implies
\begin{align}\label{eq:11}
   ((\theta+\delta)+\epsilon)n_k \geq (v-\epsilon)n_{k+1}. 
\end{align}
Thus, letting $j\in \{1,...,k-k(\epsilon,\delta,\E,\T)\}$, $(m_k)_k$, the following inequality holds, 
\begin{align}\label{eq:12}
   m_{k-j}-n_{k-j}\geq {\left(\frac{v-\epsilon}{(\theta+\delta) v+\epsilon}\right)^{j-1}}(v-\epsilon)n_k.
\end{align}
Without loss of generality, let $\epsilon<\frac{1}{2}v$. From (\ref{eq:9}), we have 
\begin{align}\label{eq:9--}
   m_k- n_k\geq (v-\epsilon)n_{k+1}>\frac{1}{2}v n_{k+1} > \frac{1}{2}v n_k .\end{align}
Therefore, by the definition of $G$,
\begin{align}
    n_{k+1}>m_k=n_k+(m_k-n_k)>( \frac{1}{2}v+1)n_k
\end{align}
which implies that $n_{k}$ is growing at least exponentially with respect to $k$. From (\ref{eq:9}) and (\ref{eq:9--}), $m_{k}$ is also growing at least exponentially. Thus, $\lim_{i\to \infty}\frac{k_i-k(\epsilon,\delta,\E,\T)}{n_{k_i}}=0$ and hence there exists $\Bar{i}\in\mathbb{N}$ such that for $i>\Bar{i}$, $\frac{k_i-k(\epsilon,\delta,\E,\T)}{n_{k_i}}<\frac{\epsilon}{8}$. We denote $\Bar{k}= \max\{k_{\Bar{i}},k(\epsilon,\delta,\E,\T)\}$, for $k_i\geq \Bar{i}$ in sequence $(k_i)_i$,
\begin{align}\label{eq:13}
    &\sum_{j=1}^{k_i} m_j-n_j-2\nonumber\\
    &\geq(v-\epsilon)n_{k_i}\left(1+\frac{v-\epsilon}{(\theta+\delta) v+\epsilon}+(\frac{v-\epsilon}{(\theta+\delta) v+\epsilon})^2+\cdots+(\frac{v-\epsilon}{(\theta+\delta) v+\epsilon})^{\Bar{k}}\right)- 2(k_i-k(\epsilon,\delta,\E,\T))\nonumber
    \\ &\geq n_{k_i}\frac{(v-\epsilon)((\theta+\delta) v+\epsilon)}{(\theta+\delta) v-v+2\epsilon} -n_{k_i}\sum_{j=k_i-\Bar{k}+1}^{\infty}(\frac{v-\epsilon}{(\theta+\delta) v+\epsilon})^i-2(k_i-k(\epsilon,\delta,\E,\T))\nonumber\\
    &\geq n_{k_i}\left(\frac{(\theta+\delta) v }{\theta+\delta-1}\right)-n_{k_i}\biggl(\frac{(v-\epsilon)((\theta+\delta) v+\epsilon)}{(\theta+\delta) v-v+2\epsilon}\Bigl(\frac{v-\epsilon}{(\theta+\delta)v+\epsilon}\Bigr)^{k_i-\Bar{k}+1}+\frac{(\theta+\delta) v }{\theta+\delta-1}\nonumber \\
    &\text{  }\text{ }-\frac{(v-\epsilon)((\theta+\delta) v+\epsilon)}{(\theta+\delta) v-v+2\epsilon}+2\frac{k_i-k(\epsilon,\delta,\E,\T)}{n_{k_i}}\biggr)\nonumber\\
    &\geq n_{k_i}\left(\frac{(\theta+\delta) v }{\theta+\delta-1}-\Bigl(\frac{1}{\theta+\delta}\Bigr)^{k_i-k(\epsilon,\delta,\E,\T)} - \epsilon\left(\frac{(\theta+\delta)^2+1}{(\theta+\delta-1)^2}+\frac{1}{2}\right)-2\frac{k_i-k(\epsilon,\delta,\E,\T)}{n_{k_i}}\right).
\end{align}
Since $\theta+\delta>\frac{1}{1-v}>1$, there exists $\hat{i}\in \mathbb{N}$ such that for all $i>\hat{i}$, we have $(\frac{1}{\theta+\delta})^{k_i- k(\epsilon,\delta,\E,\T)}<\frac{\epsilon}{4}$. Denote $\hat{k}=\max\{k_{\hat{i}},k_{\Bar{i}}\}$, for $k_i\geq \hat{k}$ together with equation (\ref{eq:13}) imply
\[\sum_{i=1}^{k_i} m_j-n_j-2 \geq n_{k_i}\left(\frac{(\theta+\delta) v }{\theta+\delta-1}-\epsilon'\right).\qedhere \]
\end{proof}
\end{lemma1}
Fix $(m_k)_k,(n_k)_k,\T,\theta,v,\delta,\epsilon',N,j,k_j$. Let us denote
\begin{align}
    &\mathcal{A}_{(m_k,n_k)_k}:=\{\E\in \NA: \E|_{n_k+1}^{m_k-1}=\T|_1^{m_k-n_k-1} \text{ for all }k\in\mathbb{N}\}.
\end{align}
For all $\E\in \mathcal{A}_{(m_k,n_k)_k} \cap \Upsilon(\T,\theta,v,\delta) $, the symbolic structure of $\E|_1^{m_{k_i}}$ can be illustrated as 
\begin{align} \label{structure0}
    \cdots\T|_1^{m_{k_1}-n_{k_1}-1}\cdots \T|_1^{m_2-n_2-1} \cdots \cdots \cdots \T|_1^{m_{k_j-1}-n_{k_j-1}-1}\dots \T|_1^{m_{k_i}-n_{k_i}-1}.
\end{align}
Let us denote
\begin{align}
&E_{k^*}{(\epsilon',\theta, v,\delta)}:=\Big\{\bigcup_{(m_k,n_k)_k}\mathcal{A}_{(m_k,n_k)_k}: (\ref{eq:13.1}) \text{ holds for all } k_i\geq k^* \Big\},\\
&L(j,k_j,N):=\{\E\in E_{k^*}{(\epsilon',\theta, v,\delta)}:n_{k_j}=N\}.
\end{align}
\begin{lemma1} \label{lemma3.2}
    Let $t:=N-\floor{N(\frac{(\theta+\delta) v }{\theta+\delta-1}-\epsilon')}$ and
    \begin{equation}
      \I^*_{\text{fixed}}:= \T|_1^{\floor{(\frac{v-\epsilon}{(\theta+\delta)v+\epsilon})^{k^*-k'} (v-\epsilon)N}-1} \cdots \T|_1^{\floor{(\frac{v-\epsilon}{(\theta+\delta)v+\epsilon}) (v-\epsilon)N}-1} \T|_1^{\floor{((\theta-\delta)v-\epsilon)N}}. \tag{$*$}  
    \end{equation}
    The Hausdorff-$s$ measure of $\pi(E_{k^*}{(\epsilon',\theta, v,\delta)})$ is bounded from above by 
    \[ \sum_{N=1}^{\infty} (c\log N)^2 N^{c\log N} \diam(\Lambda)K^{4(\floor{c\log N}-k')+3} \|f'_{\I^*_{\text{fixed}}}\|^s\sum_{\I\in \Sigma^t} \|f'_{\I}\|^s.\]
    In addition,
    \[   \dimH{\Upsilon(\T,\theta,v,\delta)}\leq \Bar{s}\]
where $\Bar{s}$ is the unique solution of
\begin{align*}
      s[\Omega^+(\T,v,\theta)+\log\frac{1}{\|f'_{\min}\|}(\delta(1+\frac{v}{(\theta-1)^2}))]+\frac{(\theta+\delta)(1-v)-1}{\theta+\delta-1}P(s)=0. 
\end{align*}

\end{lemma1}
\begin{proof}
 Notice that we can write $E_{k^*}{(\epsilon',\theta, v,\delta)}$ as a union of countable sets as 
\begin{align}\label{eq:14-}
 E_{k^*}{(\epsilon',\theta, v,\delta)}=\bigcup_{(j,k_j,N)\in \mathbb{N}^3} L(j,k_j,N).
\end{align}
Let $L(j,k_j,N):=\{\E\in E_{k^*}{(\epsilon',\theta, v,\delta)}:n_{k_j}=N\}$ be a non-empty set, from (\ref{eq:9+}), we have a natural cover of it which is the projection of cylinders of length $\floor{((\theta+\delta)v-\epsilon)N}+N$. In these cylinders, there are $k_j$ predefined terms of the form $\T|_1^{m_k-n_k-1}$ where $k\in \{k_1,\cdots,k_j\}$ . Therefore, we can notice that these cylinders all have symbolical form as
\begin{align} \label{structure}
    \cdots\T|_1^{m_{k_1}-n_{k_1}-1}\cdots \T|_1^{m_2-n_2-1} \cdots \cdots \cdots \T|_1^{m_{k_j-1}-n_{k_j-1}-1}\dots \T|_1^{\floor{((\theta-\delta)v-\epsilon)N}}.
\end{align}
From (\ref{eq:12}), we have $m_{k_j-i}-n_{k_j-i}\geq (\frac{v-\epsilon}{(\theta+\delta)v+\epsilon})^{i-1} (v-\epsilon)n_{k_j}$ which indicates that the prefixed term $\T|_1^{m_{k_j-i}-n_{k_j-i}}$ has length at least $\floor{(\frac{v-\epsilon}{(\theta+\delta)v+\epsilon})^{i-1} (v-\epsilon)n_{k_j}}-1$ for $i\in \{1,\cdots,k_j-k'\}$. Therefore, the union of cylinders of the form
 \begin{align} \label{form}
       \cdots \T|_1^{\floor{(\frac{v-\epsilon}{(\theta+\delta)v+\epsilon})^{k_j-k'} (v-\epsilon)N}-1}  \cdots  \T|_1^{\floor{(\frac{v-\epsilon}{(\theta+\delta)v+\epsilon}) (v-\epsilon)N}-1}\dots \T|_1^{\floor{((\theta-\delta)v-\epsilon)N}}
 \end{align} 
 will be a cover of $\{E_{k^*}{(\epsilon',\theta, v,\delta)}:n_{k_j}=N\}$.
 Let $\I\in \Sigma^*$ be a cylinder of length $\floor{((\theta+\delta)v-\epsilon)N}+N$ of form (\ref{form}). We denote the fixed terms $\I_{\text{fixed}}= \T|_1^{\floor{(\frac{v-\epsilon}{(\theta+\delta)v+\epsilon})^{k_j-k'} (v-\epsilon)N}-1} \cdots \T|_1^{\floor{(\frac{v-\epsilon}{(\theta+\delta)v+\epsilon}) (v-\epsilon)N}-1} \T|_1^{\floor{((\theta-\delta)v-\epsilon)N}}$ and let $\I_{\text{free}}$ to be the remaining terms combined. Then, by (\ref{6.2}) and (\ref{6.3}), we have 
 \begin{align}
    \diam(f_{\I}(\Lambda))\leq   \diam(\Lambda)  K^{4(k_j-k')+3} \|f'_{\I_{\text{fixed}}}\|\|f'_{\I_{\text{free}}}\|.
 \end{align}
 For a fixed $(n^*_k,m^*_k)_k$, we have the Hausdorff $s$-measure of $\pi(\{\E\in E_{k^*}{(\epsilon',\theta, v,\delta)}:n_{k_j}=N ,n_i=n_i^*, m_i=m_i^*\})$ for all $i\in \{1,\cdots,k_j\}\}$ is bounded by
 \begin{align}
  \diam(\Lambda)^s K^{4(k_j-k')+3} \|f'_{\I_{\text{fixed}}}\|^s\sum_{\I\in \Sigma^t} \|f'_{\I}\|^s   
 \end{align}
 where $t$ is the number of free terms and by (\ref{eq:13.1}) is equal to 
 \[N-\floor{N(\frac{(\theta+\delta) v }{\theta+\delta-1}-\epsilon')}.\]
 Since $n_k$ is growing at least exponentially with respect to $k$, we have a constant $c\in\mathbb{R}$ such that $k\leq c\log n_k$. Therefore, the Hausdorff $s$-measure of $\pi( L(j,k_j,N))=\pi(\cup_{(m_k^*,n_k^*)\in (\mathbb{N}^N)^2} \{\E\in E_{k^*}{(\epsilon',\theta, v,\delta)}:n_{k_j}=N ,n_i=n_i^*, m_i=m_i^*\})$ for all $i\in \{1,\cdots,k_j\}\}$ is bounded by
  \begin{align}\label{upper1}
  N^{c\log N} \diam(\Lambda)K^{4(k_j-k')+3} \|f'_{\I_{\text{fixed}}}\|^s\sum_{\I\in \Sigma^t} \|f'_{\I}\|^s   
 \end{align}
 which can be obtained by a simple combinatorial argument.
 In addition, we can write (\ref{eq:14-}) as 
\begin{align} \label{N-inf}
   E_{k^*}{(\epsilon',\theta, v,\delta)}=\bigcup_{N\in \mathbb{N}}\bigcup_{(j,k_j)\in \mathbb{N}^2}\{\E\in E_{k^*}{(\epsilon',\theta, v,\delta)}:n_{k_j}=N\}. 
\end{align}
To keep a consistent upper bound of the Hausdorff $s$-measure for different values of $k_j$, we denote 
\begin{align}
    \I^*_{\text{fixed}}:= \T|_1^{\floor{(\frac{v-\epsilon}{(\theta+\delta)v+\epsilon})^{k^*-k'} (v-\epsilon)N}-1} \cdots \T|_1^{\floor{(\frac{v-\epsilon}{(\theta+\delta)v+\epsilon}) (v-\epsilon)N}-1} \T|_1^{\floor{((\theta-\delta)v-\epsilon)N}}.
\end{align}
By (\ref{N-inf}) and the growth rate of $n_k$, we have the Hausdorff $s$-measure of $\pi(\bigcup_{(j,k_j)\in \mathbb{N}^2}\{\E\in E_{k^*}{(\epsilon',\theta, v,\delta)}:n_{k_j}=N\})$ is bounded by 
\begin{align}\label{upper1.}
  (c\log N)^2 N^{c\log N} \diam(\Lambda)K^{4(\floor{c\log N}-k')+3} \|f'_{\I^*_{\text{fixed}}}\|^s\sum_{\I\in \Sigma^t} \|f'_{\I}\|^s   .
 \end{align}
Thus, the Hausdorff $s$-measure of $\pi(E_{k^*}{(\epsilon',\theta, v,\delta)})$ is bounded by 
\begin{align} \label{upper2}
    \sum_{N=1}^{\infty} (c\log N)^2 N^{c\log N} \diam(\Lambda)K^{4(\floor{c\log N}-k')+3} \|f'_{\I^*_{\text{fixed}}}\|^s\sum_{\I\in \Sigma^t} \|f'_{\I}\|^s.
\end{align}
This finishes the proof of ($*$).
\end{proof}

By simply applying the root test and equation (\ref{upper2}), we have the Hausdorff dimension of $E_{k^*}{(\epsilon',\theta, v,\delta)}$ is bounded by the unique $s$ that satisfies the equation
\begin{align}\label{upper3}
    \limsup_{N\to\infty} (\frac{1}{N}\log \|f'_{\I^*_{\text{fixed}}}\|^s+\frac{1}{N}\log \sum_{\I\in \Sigma^t}\|f'_{\I}\|^s)=0.
\end{align}
The sequence $(\frac{1}{N}\log \sum_{\I\in \Sigma^t}\|f'_{\I}\|^s)_N$ is well-known being sub-additive, see Chapter 5 of \cite{Falconer}. Thus, 
\[\lim_{N\to \infty}\frac{1}{N}\log \sum_{\I\in \Sigma^t}\|f'_{\I}\|^s \text{ exists.} \]
We can rearrange the terms in (\ref{upper3}) to obtain
\begin{align}
    s\limsup_{N\to \infty}\frac{\log\|f'_{\I^*_{\text{fixed}}}\| }{N}+\left(\frac{(\theta+\delta)(1-v)-1}{\theta+\delta-1}+\epsilon'\right)P(-s\log\|f'\|)=0.
\end{align}
Notice that the number of digits in $G(N,x,\theta,v)$ is bounded by $\theta v\sum_{i=0}^{\infty}\theta^{-i}$. Together with (\ref{eq:13.1}), we have
\begin{align}
\limsup_{N\to \infty}\frac{\log\|f'_{\I^*_{\text{fixed}}}\| }{N} - \Omega^+{(\T,\theta,v)} \leq &-\log\|f'_{\min}\|\left(\delta(1+\frac{v}{(\theta-1)^2})+\epsilon\left(\frac{(\theta+\delta)^2+1}{(\theta+\delta-1)^2}+2\right)\right)
\end{align}
Notice that since $\Upsilon(\T,\theta,v,\delta)=\bigcup_{n=1}^{\infty}\bigcup_{k^*=1}^{\infty}E_{k^*}(\frac{1}{n},\theta, v,\delta)$, by basic properties of Hausdorff dimension, see for example Section 2.1 in \cite{Falconer}, we have 
\begin{align}
\dimH{\pi(\Upsilon(\T,\theta,v,\delta))}=\sup_{{(n,k^*)}\in\mathbb{N}^2}\dimH{\pi(E_{k^*}(\frac{1}{n},\theta, v,\delta))}
\end{align}
Thus, taking $n \to \infty$, we have
\begin{align} \label{estimate}
    \dimH{\pi(\Upsilon(\T,\theta,v,\delta))}\leq \Bar{s}
\end{align}
where $\Bar{s}$ is the unique solution of
\begin{align}
    s[\Omega^+(\T,v,\theta)+\log\frac{1}{\|f'_{\min}\|}(\delta(1+\frac{v}{(\theta-1)^2}))]+\frac{(\theta+\delta)(1-v)-1}{\theta+\delta-1}P(s)=0.
\end{align}
\begin{lemma1} \label{lemma3.3}
   Let $\Bar{s}$ and $s^+$ be defined as in Lemma \ref{lemma3.2} and Definition \ref{def} respectively. Then $\Bar{s}-s^+\leq O(\delta).$
    \begin{proof}
      Since $P(s)$ is differentiable, convex, strictly decreasing function with respect to $s$, we have for $\epsilon>0$, there exists $\Bar{\delta}>0$ such that for $\delta'<\Bar{\delta}$
      \[P(s+\delta')\leq P(s)+\delta'P'(s)-\epsilon.\]
      Thus, the slopes of lines that cross the origin and any point on $P(s)$ with $s\in [s^+,s^++\Bar{\delta}]$ will be bounded below by $L_1$ and bounded above by $L_2$ where
      \begin{align}
     L_1:=\frac{P(s^+)+\Bar{\delta}P'(s^+)-\epsilon}{s^++\Bar{\delta}}, \hspace{1cm}   L_2:=\frac{P(s^+)}{s^+}.
      \end{align}
    We additionally denote for convenience the slope of the line crossing the origin and $(s^+,P(s^+))$ by $L_1'$; and the slope of the line crossing the origin and $(\Bar{s},P(\Bar{s}))$ by $L_2'$, where
    \begin{align}
        L_1':&= -\frac{\theta+\delta-1}{(\theta+\delta)(1-v)-1}[\Omega^+(\T,v,\theta)-\log\|f'_{\min}\|(\delta(1+\frac{v}{(\delta+1)^2}))],\\
        L_2':&= -\frac{\theta-1}{\theta(1-v)-1}\Omega^+(\T,v,\theta).
    \end{align}
    We notice that $L_2=L_2'$. The Figure \ref{fig:intersection-of-curves1} shows that $\Bar{s}-s^+\leq \Bar{\delta}$ if $L_2'-L_1'\leq L_2-L_1$.
Substituting the value of $L_1,L_2,L_1',L_2'$ into $L_2'-L_1'\leq L_2-L_1$, we have $L_2'-L_1'\leq L_2-L_1$ when
\begin{align}\label{error} 
    \frac{\Bar{\delta}}{\delta}\geq \frac{{2 s^+}^2}{P(s^+)-s^+P'(s^+)}\cdot [&\frac{2\theta-1}{\theta-\theta v-1}\log\frac{1}{\|f'_{\min}\|}(1+\frac{v}{(\theta+1)^2})\nonumber\\
&-\frac{v}{(\theta-\theta v -1)^2}\Omega^+(\T,v,\theta)].
\end{align}
Let us denote the right hand side of inequality (\ref{error}) by $A$, then the equality case of (\ref{error}) is satisfied for $\delta= A^{-1}\Bar{\delta}$. Thus,
\[\Bar{s}-s^+\leq \Bar{\delta}\leq A\delta\leq \frac{{-4 {s^+}^2\log\|f'_{\min}\|}}{P(s^+)-s^+P'(s^+)}\cdot \frac{2\theta-1}{\theta-\theta v-1}\cdot \delta.\]
In addition, let $\hat{\delta}>0$, for $\theta\geq \frac{1}{1-v}+\hat{\delta}$, 
\begin{align}\label{NEW}
    \Bar{s}-s^+\leq \frac{{-4 \log\|f'_{\min}\|{\dimH{\Lambda}}}}{P'(\dimH{\Lambda})}\cdot \frac{2\theta-1}{\theta-\theta v-1}\cdot \delta\leq \frac{{-12\log\|f'_{\min}\| {\dimH{\Lambda}}}}{(1-v)P'(\dimH{\Lambda})}\cdot \hat{\delta}^{-1} \cdot \delta.
\end{align}
    \end{proof}
\end{lemma1}

\begin{figure}[htbp]
  \centering
  \begin{tikzpicture}
    \begin{axis}[
      width=0.6\textwidth,
      xlabel={$s$},
      legend pos=north west,
      xmin=0, xmax=3,
      ymin=0, ymax=3,
      xtick=\empty, 
      ytick=\empty, 
    ]
      \addplot[blue,mark=none, domain=0:3, samples=200, name path=curveA] {3-(7/3)*x^(1/3)};
      \node[blue, right, yshift=2mm] at (axis cs:0.3,1.5) {$P(s)$}; 

      \addplot[red,mark=none, domain=0:2, samples=100, name path=curveB] {2*x};
     
      \node[red, left, yshift=2mm] at (axis cs:1.3,2.5) {$L_2s$}; 

      \addplot[green,mark=none, domain=0:2, samples=100, name path=curveC] {(1.5)*x};
     
      \node[green, right, yshift=2mm] at (axis cs:1.44,2.5) {$L'_1s$}; 

      \addplot[black,mark=none, domain=0:4, samples=100, name path=curveD] {(1.2)*x};
     
      \node[black, right, yshift=2mm] at (axis cs:1.85,2.5) {$L_1s$}; 

      \draw [name intersections={of=curveA and curveB, by={intersection1}}];
      \node[circle,fill,inner sep=2pt] at (intersection1) {}; 
      \draw[dashed] (intersection1) -- (intersection1|-{axis cs:0,0});

      \draw [name intersections={of=curveA and curveD, by={intersection2}}];
      \node[circle,fill,inner sep=2pt] at (intersection2) {}; 
      \draw[dashed] (intersection2) -- (intersection2|-{axis cs:0,0});

      \draw [name intersections={of=curveA and curveC, by={intersection3}}];
      \node[circle,fill,inner sep=2pt] at (intersection3) {}; 
      \draw[dashed] (intersection3) -- (intersection3|-{axis cs:0,0});
    \end{axis}
 \node[below, xshift=-16mm, yshift=0mm] at (current axis.south) {$ s^+ +\Bar{\delta}$};
 \node[below, xshift=-27mm, yshift=0mm] at (current axis.south) {$ s^+ $};
 \node[below, xshift=-23.5mm, yshift=-1mm] at (current axis.south) {$ \Bar{s} $};
    
  \end{tikzpicture}
  \caption{Illustration of the relation between $s^+$ and $\Bar{s}$.}
  \label{fig:intersection-of-curves1}
\end{figure}
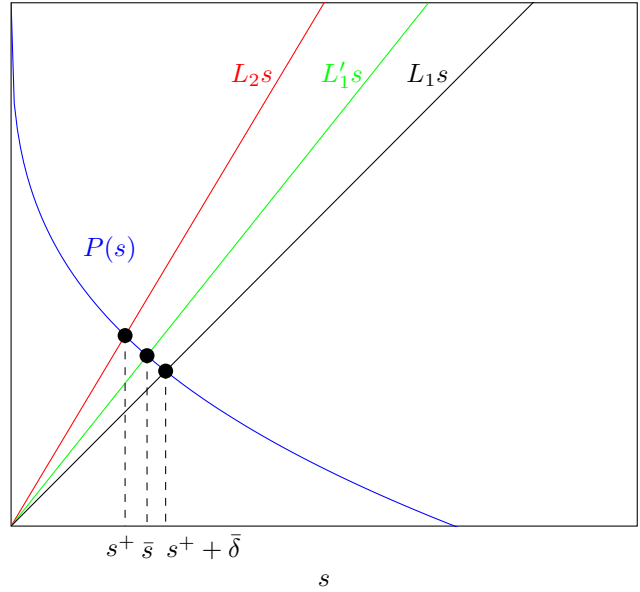

\begin{flushleft}
Applying Lemma \ref{lemma3.2} and Lemma \ref{lemma3.3}, we complete the proof of Theorem 2 part 2(1).    
\end{flushleft} 

\begin{remark}
    This part of the proof is inspired by \cite{bugeaud}. However, there is an inaccuracy in their proof. By their way of abuse of notations, the inequality (2.6) in \cite{bugeaud}, the subsequence $(k_j)_j$ has to be successive integers, which is obviously not the case in general. Even though it seems like a critical error, it turns out that if we just write the proof without the abuse of notations, the method still functions properly. As one can see from the proof above, the corrected method of proof looks much more complicated than the proof presented in \cite{bugeaud}.
\end{remark}

\subsection{A proper measure and the local dimension: lower bound}
\begin{flushleft}
   To find the lower bound of $\dimH{\pi(\Delta(\T,v,\theta))}$, we will construct a measure on a subset of it. 
Recall from Definition $\ref{def}$ that $\Delta(\T,\theta,v):= \{\E\in\NA:v_e(\E,\T)\geq v\} \cap \{\E\in\NA:v_s(\E,\T)\geq \theta v\}$. 
Here we aim to build two sequences $\{m_k\}_k$ and $\{n_k\}_k$, such that they simultaneously satisfy the following three conditions, 
\end{flushleft}
\begin{align}
    &\lim_{k\to \infty} \frac{m_k-n_k}{n_{k+1}}=v, \label{eq:37} \\
    &\lim_{k\to \infty} \frac{m_k-n_k}{n_{k}}=\theta v, \label{eq:38}\\
    &n_k<m_k<n_{k+1}. \label{eq:39}
\end{align}
Let $L$ be the set that contains all elements $\E\in \NA$ such that $ \E|_{n_k+1}^{m_k-1} = \T|_1^{m_k-n_k-1}$. Equations
(\ref{eq:37}) and (\ref{eq:38}) indicate an easy choice:
\begin{align}
n_k=\floor{a\theta^k}\text{ and } m_k=\floor{(\theta v+1) n_k},    
\end{align}
where $a\in \mathbb{R}^+$ is large enough. More precisely, to validly define the set $L$, we need $a$ satisfying 
\[n_{k+1}-m_k\geq 2 \text{ and } m_k-n_k\geq 2 \text{ for all }k\] 
which is guaranteed when both of the following inequalities are satisfied
\[(a\theta^{k+1}-1)-((\theta v+1)(a\theta^k+1)+1)\geq 2 \text{ for all }k,\]
\[((\theta v+1)(a\theta^k-1)-1)-(a\theta^k+1)\geq 2 \text{ for all }k.\]
Thus, for $a\geq \max\{\frac{5+\theta v}{(\theta-\theta v-1)\theta},\frac{5+\theta v}{\theta^2v}\}$, set $L$ is well defined. It is not hard to check that $L\subseteq \Delta(\T,v,\theta) $.

The following steps prove for $\T\in G$, $v\in(0,1)$ and $\theta\in(\frac{1}{1-v},+\infty)$ that $\dimH{\pi(L)}\geq \min\{s^-(\T,v,\theta), d\}$, where $s^-$ can be found in Definition $\ref{def}$. By constructing a proper measure on $L$, our desired lower bound follows from the use of the mass distribution principle, see for example Chapter 10 of \cite{Falconer}.

For $l\in\mathbb{N}$, let $\widehat{\Sigma}^l:=\{\I^*\in\Sigma^l| \I^*\J \in L$ for some $\J\in\NA\}$ and fix any arbitrary $\E_0\in L$. We define discrete measures $\mu_l$ by assigning measures to finite discrete points $\pi(\E|_1^{l}\E_0|_{l+1}^{\infty})$ for all $\E\in L$ in the following manner:
\begin{align}\label{eq:32.}
    \mu_l(\{\pi(\E|_1^{l}\E_0|_{l+1}^{\infty})\})= \frac{\|f'_{\E|_1^{l}}\|^{s^-(\T,v,\theta)}}{\sum_{\I^*\in \widehat{\Sigma}^l } \|f'_{\I^*}\|^{s^-(\T,v,\theta)}           }.
\end{align}
Notice that $\pi(L)$ is compact. To see this, let $x\in\pi(L) $ and for a sequence $(x_i)_i$ where $x_i\in \pi(L)$ such that
\[\lim_{i\to\infty} d(x,x_i)=0.\]
From OSC of the IFS, we know that $\pi^{-1}(x)\cap L$ is not empty. Therefore, $x\in \pi(L)$ and then $\pi(L)$ is closed. It is also easy to obtain that $\pi(L)$ is bounded as $\Lambda$ is bounded. By Proposition 1.9 in \cite{Falconer}, this extends to a probability measure $\mu$ supported by $\pi(L)$ which is a weak limit of a subsequence $(l_k)_k$ of the measures $\mu_{l_k}$. Notice that, for any $\E\in L$ and $k> k'$, let $\widehat{\Sigma}_{k'}^{k}(\E):= \{\J\in \widehat{\Sigma}^k|\J|_1^{k'}=\E|_1^{k'} \}$, by (\ref{6.3}) and (\ref{eq:32.}),
\begin{align}\label{eq:35.}
    \mu_k(\pi([\E|_1^{k'}]))= \frac{\sum_{\J\in \widehat{\Sigma}_{k'}^{k}} \|f'_{\J}\|^{s^-}  }{ \sum_{\I\in \widehat{\Sigma}^{k}}\|f'_{\I}\|^{s^-}  } \leq K^2 \frac{\|f'_{\E|_1^{k'}}\|^{s^-}}{\sum_{\I\in \widehat{\Sigma}^{k'}}\|f'_{\I}\|^{s^-} }.
\end{align}

\begin{lemma1}\label{3..}
Let $\{f_i\}_{i=1}^N$ be a conformal iterated function system that satisfies the open set condition and $\Lambda$ is the self-conformal invariant set. Let $L$ be defined as above, $\E_0\in L$. Let $\mu$ be the weak limit of $(\mu_{l_k})_k$ as above. For the $s^-(\T,v,\theta)$ from Definition $\ref{def}$, for any $\E\in L$ we have
\begin{align} \label{3.4.2}
    \liminf_{r\to 0} \frac{\log\mu(B(\pi(\E),r))}{\log|r|}\geq s^-(\T,v,\theta).
\end{align}
\begin{proof}
By (\ref{6.4}), we can find a open set $O\subset V$ such that $\Lambda\subset\Bar{O}\subset V$. Let $\E\in L$, $B(\pi(\E),r)$ be a ball of radius $r<1$. Curtail each infinite sequence $i_1 i_2 \cdots\in \Sigma^{\mathbb{N}}$ after the first term $i_p$ such that 
\begin{align}
    \|f'_{\min}\|r\leq \diam(f_{i_1\cdots i_p}(O)) \leq r.
\end{align}
We denote $Q_r\subset \Sigma^*$ to be the finite set of all finite sequences obtained in this way. Thus, for every $\I\in \Sigma^{\mathbb{N}} $, there is exactly one value of $p$ such that $i_1\cdots i_p\in Q_r$. It is not hard to notice that $\{f_{i_1\cdots i_p}(O):i_1\cdots i_p\in Q_r\}$ is disjoint. We choose $a_1$ and $a_2$ so that $O$ contains a ball of radius $a_1$ and contained in a ball of radius $a_2$. By (\ref{6.1}) and (\ref{6.2}), we obtain that $f_{i_1\cdots i_p}(O)$ is contained in a ball of radius $\|f'_{i_1\cdots i_p}\|a_2$, so a ball of radius $\frac{K' a_2}{\diam(O)} r$; $f_{i_1\cdots i_p}(O)$ contains a ball of radius $K'^{-1}\|f'_{i_1\cdots i_p}\|a_1$, so a ball of radius $\frac{\|f'_{\min}\|a_1}{K' \diam(O)}r$. Let $\widehat{Q}_r(B)$ be a set of finite sequences $i_1\cdots i_p\in Q_r$ such that $B$ intersect $f_{i_1\cdots i_p}(\Bar{O})$. By Lemma 9.2 in \cite{books/daglib/0067317}, $\widehat{Q}_r(B)$ contains at most $q=(1+2\frac{K' a_2}{\diam(O)})^d(\frac{\|f'_{\min}\|a_1}{K' \diam(O)})^{-d}$ elements where $d$ is the topological dimension of the domain of functions in the IFS. Thus, 
\begin{align} \label{3.4.3}
    \mu(B(\pi(\E),r))=\mu(B(\pi(\E),r)\cap\Lambda)\leq \mu\{\pi(\cup_{\I\in \widehat{Q}_r(B)} \I )\}.
\end{align}
From the definition of $\widehat{Q}_r(B)$, for all $r<1$, there is an integer $k'(r)$ such that $\E|_1^{k'(r)}\in \widehat{Q}_r(B)$. We denote $k^*(r,\E,\I)$ as the length of $\I\in \widehat{Q}_r(B)$ and it is bounded by $\frac{\log\|f'_{\min}\|}{\log\|f'_{\max}\|}k'(r)$ and $\frac{\log\|f'_{\max}\|}{\log\|f'_{\min}\|}k'(r)$. For convenience, we write $k'(r)$ as $k'$ and $k^*(r,\E,\I)$ as $k^*$. Notice that $r\to 0$ implies $k',k^*\to\infty$. In addition, for all $\Bar{k}\in \mathbb{N}$, there exists $r\in \mathbb{R}$ such that $k'(r)=\Bar{k}$. Therefore, 
\begin{align}
    \liminf_{r\to 0}\frac{\log \mu(B(\pi(\E),r))}{\log|r|}&\geq \liminf_{r\to 0}\frac{\sum_{\I\in \widehat{Q}_r(B)}\mu(\I)}{\log|r|}\\
    &\geq \liminf_{r\to 0} \frac{\log q r^{s^-} / \min_{k^*}\sum_{\I\in \widehat{\Sigma}^{k^*}}\|f'_{\I}\|^{s^-}}{\log {r}}\\
    &\geq s^- + \liminf_{k'\to \infty} \frac{-\log \min_{k^*}\sum_{\I\in \widehat{\Sigma}^{k^*}}\|f'_{\I}\|^{s^-} }{\log\frac{\diam(O)\|f'_{\E|_1^{k'}}\|}{\|f'_{\min}\|K}}.
\end{align}
Since for a given $k'$ the number of elements in $\widehat{Q}_r(B)$ is bounded. We define $\hat{k}_{k'}(r,\E)$ to be the $k^*$ that obtains
the extreme of $\min_{k^*}\sum_{\I\in \widehat{\Sigma}^{k^*}}\|f'_{\I}\|^{s^-}$. Then
\begin{align}
        s^- + \liminf_{k'\to \infty} \frac{-\log \inf_{k^*}\sum_{\I\in \widehat{\Sigma}^{k^*}}\|f'_{\I}\|^{s^-} }{\log\frac{\diam(O)\|f'_{\E|_1^{k'}}\|}{\|f'_{\min}\|K}}
    &= s^- - \limsup_{k'\to\infty} \frac{\log \Sigma_{\I\in \widehat{\Sigma}^{\hat{k}}} \|f'_{\I}\|^{s^-}}{\log \|f'_{\E|_1^{k'}}\|}\\
    &= s^-- \limsup_{k'\to\infty} \frac{\hat{k}}{k'} \frac{(1/\hat{k})\log \Sigma_{\I\in \widehat{\Sigma}^{\hat{k}}} \|f'_{\I}\|^{s^-}}{(1/k')\log \|f'_{\E|_1^{k'}}\|}.
\end{align}
Notice that $(1/k')\log \|f'_{\E|_1^{k'}}\|$ is a negative value bounded by $\log\|f'_{\min}\|$ and $\log\|f'_{\max}\|$, and $\frac{\hat{k}_{k'}}{k'}$ is a positive value bounded by $\frac{\log\|f'_{\min}\|}{\log\|f'_{\max}\|}$ and $\frac{\log\|f'_{\max}\|}{\log\|f'_{\min}\|}$. Thus, a sufficient condition to finish the proof is that for $k'$ large enough $(1/{\hat{k}_{k'}})\log \sum_{\I\in \widehat{\Sigma}^{\hat{k}_{k'}} } \|f'_\I\|^{s^-}\geq 0$.
We assume $\liminf_{k'\to\infty}(1/{\hat{k}_{k'}})\log \sum_{\I\in \widehat{\Sigma}^{\hat{k}_{k'}} } \|f'_\I\|^{s^-}<0$ for convenience of writing the following calculation. 
Notice that $(\hat{k}_{k'}(r,\E))_{k'}$ is a non-decreasing sequence of nature numbers. We further divide $(\hat{k}_{k'}(r,\E))_{k'}$ into two subsequences $(\hat{k}^+_{k'})_{k'}$ and $(\hat{k}^-_{k'})_{k'}$, such that for all $k'\in\mathbb{N}$, there exist $n_{k_{(+,k')}},n_{k_{(-,k')}}\in\mathbb{N}$, $n_{k_{(+,k')}}<\hat{k}^+_{k'}\leq m_{k_{(+,k')}}$ and $m_{k_{(-,k')}}\leq\hat{k}^-_{k'}< n_{k_{(1,k')}+1}$. In addition, recall definition of $s^-$ in Definition $\ref{def-}$, 
\begin{align}\label{eql1}
 \liminf_{k'\to\infty}(1/{\hat{k}^+_{k'}})\log \sum_{\I\in \widehat{\Sigma}^{\hat{k}^+_{k'}} } \|f'_\I\|^{s^-} &\geq \limsup_{k'\to\infty}  \frac{n_{k_{(+,k')}}}{\hat{k}^+_{k'}} \cdot \liminf_{k'\to\infty} \Big(\frac{\log\|f'_{G(n_k,x,\theta,v)}\|}{n_k} \nonumber\\
 & \text{ }\text{ }  +\frac{\theta-\theta v-1}{\theta-1}P(s^{-})\Big)+\liminf_{k'\to\infty}\frac{\hat{k}^+_{k'}-m_k}{\hat{k}^+_{k'}}\log\|f'_{\max}\|\nonumber\\
&\geq \liminf_{k'\to\infty} \Big(\frac{\log\|f'_{G(n_k,x,\theta,v)}\|}{n_k}+\frac{\theta-\theta v-1}{\theta-1}P(s^{-})\Big)\nonumber \\
&\geq \liminf_{M\to\infty} \Big(\frac{\log\|f'_{G(M,\T,\theta,v)}\|}{M}+\frac{\theta-\theta v-1}{\theta-1}P(s^{-})\Big)=0,
\end{align}
and as $s^-\leq\dimH{\Lambda}$
\begin{align}\label{eql2}
 \liminf_{k'\to\infty}(1/{\hat{k}^-_{k'}})\log \sum_{\I\in \widehat{\Sigma}^{\hat{k}^-_{k'}} } \|f'_\I\|^{s^-} &\geq \limsup_{k'\to\infty}\frac{n_k}{\hat{k}^-_{k'}} \cdot \liminf_{k'\to\infty} \Big(\frac{\log\|f'_{G(n_k,x,\theta,v)}\|}{n_k} \nonumber \\
 &\text{ }\text{ }+\frac{\theta-\theta v-1}{\theta-1}P(s^{-})\Big)+\liminf_{k'\to\infty}\frac{\hat{k}-m_k}{\hat{k}}P(s^{-}) \nonumber \\
 & \geq(1-v) \liminf_{k'\to\infty} \Big(\frac{\log\|f'_{G(n_k,x,\theta,v)}\|}{n_k}+\frac{\theta-\theta v-1}{\theta-1}P(s^{-})\Big)\nonumber \\
 &\geq(1-v) \liminf_{M\to\infty} \Big(\frac{\log\|f'_{G(M,\T,\theta,v)}\|}{M}+\frac{\theta-\theta v-1}{\theta-1}P(s^{-})\big)\geq 0.
\end{align}
By inequalities \ref{eql1} and \ref{eql2},
\begin{align*}
    \liminf_{k'\to\infty}(1/{\hat{k}_{k'}})\log \sum_{\I\in \widehat{\Sigma}^{\hat{k}} } \|f'_\I\|^{s^-} &\geq\min \{ \liminf_{k'\to\infty}(1/{\hat{k}^-_{k'}})\log \sum_{\I\in \widehat{\Sigma}^{\hat{k}^-_{k'}} } \|f'_\I\|^{s^-}, \liminf_{k'\to\infty}(1/{\hat{k}^-_{k'}})\log \sum_{\I\in \widehat{\Sigma}^{\hat{k}^-_{k'}} } \|f'_\I\|^{s^-}\}\\
    &\geq 0.
\end{align*}
Thus, by contradiction, it implies 
\[-\limsup_{k'\to\infty} \frac{\log \sum_{\I\in \widehat{\Sigma}^{\hat{k}^-_{k'}}} \|f'_{\I}\|^{s^-}}{\log \|f'_{\E|_1^{\hat{k}^-_{k'}}}\|}\geq 0.\] 
This finishes the proof of (\ref{3.4.2}). \end{proof}
\end{lemma1}
In order to compute the lower bound of $\dimH{\Delta(\T,v,\theta)}$, we apply Lemma \ref{3.4.2} and the Mass Distribution Principle, see Proposition 2.3 in \cite{Falconer}. We obtain  
\begin{align}
    \dimH{\Delta(\T,v,\theta)}\geq \dimH{L} \geq \min\{s^-,\dimH{\Lambda}\}.
\end{align}
This finishes the proof of Theorem 2 part 2(2).

\section{Hausdorff dimension of eventually always hitting set: Proof of Theorem 1}
In this section, we combine all results discussed in previous sections to prove Theorem 1.

\begin{lemma1}
 \label{Lemma:5.2}
Let $\T\in G$, $v_a,v_b\in \mathbb{R}\cup \{\infty\}$ and $(a_n)_n,(b_n)_n$ be two sequences of natural numbers such that $\lim_{n\to\infty}\frac{a_n}{n}=v_a$ and $\lim_{n\to\infty}\frac{b_n}{n}=v_b$. If $v_b< v_a$, then
\[R_e((a_n)_n,\T) \subset R_e((b_n)_n,\T) .\]
\begin{proof} 
We can find a large enough $N\in\mathbb{N}$, such that for all $n>N$, we have $a_n>b_n$. Thus, for any $\E\in R_e((a_n)_n,\T) $, we have for large enough $N'\in \mathbb{N}$, that for all $n>N'$ there exists $n'\leq n$ such that $[\E|_{n'+1}^{n'+a_n}]=[\T|_{1}^{a_n}]$. Therefore, for $n> \max\{N,N'\}$, we have $\E|_{n'+1}^{n'+b_n}=\T|_{1}^{b_n}$, which implies $\E \in R_e((b_n)_n,x)$.   \end{proof}
\end{lemma1}
It is not hard to obtain a “sandwich rule” here. Let $v_b<v_c<v_a$. By directly using Lemma \ref{Lemma:5.2}, for a sequence $(c_n)_n$ such that $\lim_{n\to\infty} \frac{c_n}{n}=v_c$, 
\begin{align}\label{sandwich}
   \dimH{\pi(R_e(a_n)_n,\T))}<\dimH{\pi(R_e(c_n)_n,\T))}<\dimH{\pi(R_e(b_n)_n,\T))}. 
\end{align}
\begin{lemma1} \label{lemma: 5.1}
Theorem 1 holds for $a_n=\floor{v n}$, when $v\in (0,1)$.
\end{lemma1}
\begin{proof}
For a given $v\in (0,1)$ and $\E\in\NA$, denote $g(v,\theta,\T):=-\frac{\theta-1}{\theta-\theta v-1}\Omega^+(\T,v,\theta)$. Recall from Definition \ref{def-}, $g(v,\theta,\T)$ is the linear part of the right hand side of the equation $P(s^+(\T,v,\theta))= -s^+(\T,v,\theta)\frac{\theta-1}{\theta-\theta v-1}\Omega^+(\T,v,\theta)$. It is not hard to obtain that $  \Omega^+(\T,v,\theta)\leq \frac{\theta^2 v}{\theta-1}\log\|f'_{\max}\|$, and so $g(v,\theta,\T)\geq \frac{\theta^2 v}{\theta-\theta v -1}\log\frac{1}{\|f'_{\max}\|}$. By calculating the derivative with respect to $\theta$, $\frac{\theta^2v}{\theta-\theta v-1}\leq \frac{4v}{(1-v)^2}$ and the equality is achieved when $\theta=\frac{2}{1-v}$. Thus, we get $g(v,\theta,\T)\geq \frac{4 v}{(1-v)^2}\log\frac{1}{\|f'_{\max}\|}$. The existence of the lower bound implies the existence of $\inf_{\theta\in (\frac{1}{1-v},+\infty)} g(v,\theta)$, and we define $h(\T,v):=\inf_{\theta\in (\frac{1}{1-v},+\infty)} g(v,\theta)$. 

Now we will prove that there exists $\hat{\delta}>0$ such that $  g(v,\theta,\T)> h(\T,v) $ for all $\theta\in (\frac{1}{1-v},\frac{1}{1-v}+\hat{\delta}]$. Notice that $  \Omega^+(\T,v,\theta)\geq \frac{\theta^2 v}{\theta-1}\log\|f'_{\min}\|$, and so $h(\T,v) \leq g(v,\frac{2}{1-v},\T)\leq \frac{4}{(1-v)^2}\log\frac{1}{\|f'_{\min}\|} $. Let $\hat{\delta}=\frac{2}{(1-v)v}\cdot \frac{\log \|f'_{\min}\|}{\log \|f'_{\max}\|}$, and $\theta\in (\frac{1}{1-v},\frac{1}{1-v}+\hat{\delta}]$. Then,
\[g(v,\theta,\T)\geq \frac{\theta^2v}{\theta-\theta v-1}\cdot\log\frac{1}{\|f'_{\max}\|}\geq \frac{v}{1-v}\cdot \hat{\delta}^{-1} > \frac{4}{(1-v)^2}\log\frac{1}{\|f'_{\min}\|}\geq h(\T,v).\]
Therefore, $\inf_{\theta\in (\frac{1}{1-v},+\infty)} g(v,\theta)=\inf_{\theta\in (\frac{1}{1-v}+\hat{\delta},+\infty)} g(v,\theta) $. We either have a sequence $(\theta_n)_n$ where $\theta_n\in (\frac{1}{1-v}+\hat{\delta})$ for all $n\in\mathbb{N}$ such that $\lim_{n\to\infty}g(v,\theta_n)= h(\T,v)$ or have a $\theta^*\in (\frac{1}{1-v}+\hat{\delta},+\infty)$ such that $h(v,\theta^*)=h(\T,v)$. In both of these cases, we can construct a countable dense subset $\Xi_{\hat{\delta}} \subset (\frac{1}{1-v}+\hat{\delta},+\infty)$ containing $\{\theta_n:n\in\mathbb{N}\}$ or $\theta^*$ such that for any $\delta>0$, by Theorem 2 part(2)1, we have $\theta\in \Xi$ such that $\dimH{\Upsilon(\T, v,\theta ,\delta)}\leq s^+(\T,v,\theta) + O(\delta)$ for all $\delta$ small and by Lemma \ref{lemma3.3} $O(\delta)$ is uniform for $\theta\in (\frac{1}{1-v}+\hat{\delta},\infty)$. The purpose of constructing $\Xi_{\hat{\delta}}$ is to find a countable subset of $(\frac{1}{1-v}+\hat{\delta},\infty)$ such that there exists a sequence of a point that attains the extreme of $\dimH{\pi(\Upsilon(\T,v,\theta,\delta))}$ for $\theta\in (\frac{1}{1-v}+\hat{\delta},\infty)$.
For a given $\epsilon>0$, we take different small $\delta=\delta_{\hat{\delta}}$ for each $\hat{\delta}>0$ such that $O(\delta_{\hat{\delta}})\leq \epsilon $. Recall from Definition $\ref{def-}$, $\hat{s}^+(\T,v):= \sup_{\theta\in (\frac{1}{1-v},+\infty)} s^+(\T,v,\theta) $. Since $\Xi_{\hat{\delta}}$ contains $(\theta_n)_n$ or $\theta^*$ that attains the extreme of $s^+(\T,v,\theta)$.  Taking $\epsilon\to 0$ we can conclude that for a given $\hat{\delta}>0$
\begin{align}
    \dimH{\pi(\bigcup_{\theta\in \Xi_{\hat{\delta}}}  \Upsilon(\T,\theta,v,\delta_{\hat{\delta}}))  } \leq \sup_{\theta\in \Xi_{\hat{\delta}}} s^+(\T,v,\theta) =\sup_{\theta\in (\frac{1}{1-v}+\hat{\delta},\infty)} s^+(\T,v,\theta).
\end{align}
Therefore,
\begin{align}\label{eq:51}
    \dimH{\pi(R_e(\T,v))} \leq \dimH{\bigcup_{\hat{\delta}\in \{\frac{1}{n}:n\in \mathbb{N}\}} \pi(\bigcup_{t\in \Xi}  \Upsilon(x,t,v,\delta_\theta))  } \leq \lim_{n\to\infty}\sup_{\theta\in (\frac{1}{1-v}+\frac{1}{n}),\infty} s^+(\T,v,\theta)
= \hat{s}^+(\T,v).
\end{align}
 Since $  \Omega^-(\T,v,\theta)\geq \frac{\theta^2 v}{\theta-1}\log\|f'_{\min}\|$. Following a similar argument as previous, together with Theorem 2 part (2)1, we can deduce that
\begin{align}\label{eq:52}
    \dimH{\pi(R_e(\T,v)})\geq \dimH{\pi(\bigcup_{t\in \Xi} \Delta(\T,v,\theta) )}\geq\sup_{\theta\in \Xi} s^-(\T,v,\theta) = \hat{s}^-(\T,v).
\end{align}
Thus, followed by equations (\ref{eq:51}) and (\ref{eq:52}), we get 
\begin{align}\label{sandwich1}
 \hat{s}^-(\T,v)\leq\dimH{\pi(R_e(\T,v))}\leq \hat{s}^+(\T,v).   
\end{align}

\end{proof}

We inherit notations from Lemma \ref{lemma: 5.1} in the following Lemma, and provide a proof of continuity of the upper and lower bound of $\dimH{\pi(R_e(\T,v))}$.

\begin{lemma1} \label{Lemma:5.3}
Let $x\in\Lambda$ and $v\in [0,1]$. Let $w^+(\T,v)$ and $w^-(\T,v)$ be as in Definition \ref{def-}. Both of $w^+(\T,v)$ and $w^-(\T,v)$ are continuous with respect to $v$.
\end{lemma1}
\begin{proof}
We give a proof for $w^+(\T,v)$. The claim for $w^-(\T,v)$ follows similarly. We first prove for $v\in (0,1)$ and $\theta\in(\frac{1}{1-v},\infty)$ that $g(v,\theta,\T)$ is continuous with respect to $v\in(0,1)$. Since $\frac{\theta-1}{\theta-\theta v-1}$ is continuous with respect to $v\in(0,1)$, it is sufficient to prove that $\Omega^+(\T,v,\theta)$ is continuous with respect to $v\in(0,1)$. That is, for a given $v\in(0,1)$, for any $\epsilon>0$, there exist $\delta>0$, such that for all $v'\in (v-\delta,v+\delta)$, we have $|\omega^+(\T, \theta,v)-\omega^+(\T, \theta,v')|<\epsilon$ when $\theta> \max\{\frac{1}{1-v},\frac{1}{1-v'}\}$. Recall from Definition $\ref{def-}$ that 
\[\Omega^+(\T,v,\theta)=\limsup_{M\to\infty} \frac{\log \left\|f'_{\T|_1^{\floor{\frac{v M}{\theta^p}}}\dots \T|_1^{\floor{v M \theta}}}\right\|}{M}=\limsup_{M\to\infty}\frac{\log\|f'_{G(M,\T,\theta,v)}\|}{M}.\]
We will first prove continuity from the right. Left continuity follows from similar method. Recall that $p$ is the largest integer such that $\frac{vM}{p}\geq 1$. For each $M\in\mathbb{N}$, $\delta'>0$, let $p^*$ be the largest integer such that $\floor{\frac{(v+\delta') M}{\theta^{p^*} } }-\floor{\frac{v M}{\theta^{p^*} }}\geq 1$. For every $-1\leq j\leq p^*$, we can split $\I|_{1}^{\floor{\frac{(v+\delta') M}{\theta^j }}}$ into $\mathbf{o}_j=\I|_{1}^{\floor{\frac{v M}{\theta^j }}}$ and $\R_j=\I|_{\floor{\frac{v M}{\theta^j }}+1}^{\floor{\frac{(v+\delta') M}{\theta^j }}}$. For $M$ large enough, we have $p^*<p$, and define $\mathbf{u}:= \I|_{1}^{\floor{\frac{v M}{\theta^p}}}\dots \I|_{1}^{\floor{\frac{v M}{\theta^{p^*}}}}$. By applying (\ref{6.3}), we have
\begin{align}\label{eq:59}
    |\Omega(\T,v+\delta',\theta)- \Omega(\T,v,\theta)| =-\limsup_{M\to\infty} \frac { \log\|f'_{\mathbf{u}}\|+ \sum_{j=-1}^{p^*} \log\|f'_{\R_j}\|}{M} .
\end{align}
Notice that $\floor{\frac{(v+\delta') M}{\theta^{p^*} } }-\floor{\frac{v M}{\theta^{p^*} }}\geq \frac{(v+\delta') M}{\theta^{p^*} } -\frac{v M}{\theta^{p^*} }-1
\geq \frac{\delta'M}{\theta^*} -1$. If $\frac{\delta'M}{p^*}>\theta+1$, $p^*$ is not the largest integer such that $\floor{\frac{(v+\delta') M}{\theta^{p^*} } }-\floor{\frac{v M}{\theta^{p^*} }}\geq 1$, which contradicts its definition. Therefore, $\frac{\delta'M}{p^*}\leq \theta+1$. Together with $\floor{\frac{(v+\delta') M}{\theta^{p^*} } }-\floor{\frac{v M}{\theta^{p^*} }}\leq \frac{\delta'M}{\theta^*} +1$, we have
\begin{align}
    \floor{\frac{(v+\delta') M}{\theta^{p^*} } }-\floor{\frac{v M}{\theta^{p^*} }}\leq \theta+2.
\end{align}
Meanwhile, notice that by $p^*<p$ and the definition of $p$, we get
\begin{align}
    \floor{\frac{(v+\delta') M}{\theta^{p^*} } }\leq \floor{\frac{v M}{\theta^{p^*} }}+\theta+2 \leq 2\theta+2,
\end{align}
which implies 
\begin{align}\label{eq:60}
    -\limsup_{M\to\infty} \frac{\log\|f'_\mathbf{u}\|}{M} \leq \lim_{M\to\infty}\frac{ (2\theta+2)(\frac{1}{\delta'}+\frac{1}{v})(\sum_{i=0}^{\infty}\frac{1}{\theta^i})}{M}\log\frac{1}{\|f'_{\min}\|}=0.
\end{align}
As $|\R_j|\leq \frac{\delta' M}{\theta^j}+2$, we get $\sum_{j=-1}^{p^*}|\R_j|\leq \frac{\theta}{\theta-1}\delta'M+2p^*$, thus
\begin{align}\label{eq:61}
   -\limsup_{M\to\infty} \frac{\sum_{j=-1}^{p^*} \log\|f'_{\R_j}\|}{M}\leq \lim_{M\to\infty} \frac{\frac{\theta}{\theta-1}\delta'M+2p^*}{M}\log\frac{1}{\|f'_{\min}\|} \leq \frac{\theta\delta'}{\theta-1}.
\end{align}
Substituting estimates (\ref{eq:60}) and (\ref{eq:61}) into (\ref{eq:59}), we get the desired right continuity, and left continuity can be obtained by similar method.

We now want to check the continuity of $h(\T,v)$. The strategy is to show that the $\theta_1,\theta_2\in (\frac{1}{1-v},\infty)$ attaining the extreme of $g(\theta,v,\T)$ and $g(\theta,v',\T)$ respectively are close when $v$ and $v'$ are close. By the continuity of $g(\theta,v,\T)$, for $\delta> 0$, $v\in (0,1)$ and $v'\in (v-\delta,v+\delta)\subset(0,1)$, we have 
\begin{align}\label{eq:62}
   \inf_{\theta\in (\frac{1}{1-v^+},+\infty)} g(\theta,v,\T)-\inf_{\theta\in (\frac{1}{1-v^+},+\infty)} g(\theta,v',\T)=O(\delta) 
\end{align}
where $v^+=\max\{v,v'\}$. We now check that both $h(\T,v)$ and $h(\T,v')$ can be obtained by $g(\theta,v,\T)$ and $g(\theta,v',\T)$ respectively for $\theta\in (\frac{1}{1-v^+},\infty)$. We substitute $\theta=\frac{2}{1-v}$ into $g(\theta,v,\T)$, then
\[g(\theta,v,\T)\leq\frac{4v}{(1-v)^2}\log \frac{1}{\|f'_{\min}\|}.\] 
Let $v^-=\min\{v,v'\}$. If $\theta\in (\frac{1}{1-v^-},\frac{1}{1-v^+})$, we have $g(\theta,v,\T)\geq \frac{v}{\delta}\log\frac{1}{\|f'_{\max}\|}\geq g(\theta,v,\T)\geq\frac{4v}{(1-v)^2}\log \frac{1}{\|f'_{\min}\|}> h(\T,v)$ which implies $g(\theta,v,\T)$ cannot approach $h(\T,v)$ for $\theta\notin (\frac{1}{1-v^+},\infty)$. Thus, both $g(\theta,v,\T)$ and $g(\theta,v',\T)$ attain their infimums at the same $\theta\in (\frac{1}{1-v^+})$.
Therefore, (\ref{eq:62}) implies the continuity of $h(\T,v)$ for $v\in (0,1)$. 

The continuity of $s^+(\T,v)$ follows easily from the continuity of $h(\T,v)$. The continuity will remain if we take maximum of two two continuous functions, so $\omega^+(\T,v)$ is continuous. It is not hard to notice that $\lim_{v\to 0^+} \omega^+(\T,v)= \dimH{\Lambda}$ and   $\lim_{v\to 1^-} \omega^+(\T,v)= 0$, which completes the proof.
\end{proof} 
By continuity of $w^+(\T,v)$ and $w^-(\T,v)$ for $v\in(0,1)\subset [0,1]$ and equation $\ref{sandwich}$, for a sequence $(c_n)_n$ such that $\lim_{n\to\infty} c_n= v $,
\begin{align}
    s^-(\T,v)\leq\dimH{\pi(R_e((c_n)_n,n),\T))}\leq s^+(\T,v).
\end{align}
That finishes the proof of Theorem 1 for $v\in(0,1)$. We finish the proof of Theorem 1 by checking the remaining values of $v$, namely $v=0,1$.
\begin{proof}[Proof of Case 1 in Theorem 1]
For a given $\T\in G$, $\epsilon>0$ and any sequence of natural numbers $(a_n)_n$ with $\lim_{n\to\infty} \frac{a_n}{n}=0$, by Lemma \ref{lemma: 5.1} and Lemma \ref{Lemma:5.3}, we have 
\[\dimH{\pi(R_e((a_n)_n,\T)) } \geq \hat{s}^-(\T,\epsilon),\]
thus 
\[\dimH{\pi(R_e((a_n)_n,\T) )} \geq \lim_{\epsilon\to 0} \hat{s}^-(\T,\epsilon)= \omega^-(\T,0). \]
Since $\pi(R_e((a_n)_n,\T))\subset \Lambda$, we have
\[\dimH{\pi(R_e((a_n)_n,\T)) } \leq \dimH{\Lambda}= \omega^+(\T,0)=\omega^-(\T,0) .\]
The argument for $\lim_{n\to\infty} \frac{a_n}{n}=1$ is similar, we omit it here.
\end{proof}

\section{Applications and explicit results}

Theorem 1 gives lower and upper bounds that are not sharp. However, the specific condition $\Omega^+{(\T,\theta,v)}=\Omega^-{(\T,\theta,v)}$ mentioned in Theorem 1 is satisfied in many cases.

\begin{corollary}
 Let $\Lambda$ be a self-similar set satisfying open set condition, and $\T \in G$ be periodic. Then the specific condition $\Omega^+{(\T,\theta,v)}=\Omega^-{(\T,\theta,v)}$ in Theorem 1 is satisfied.
\end{corollary}

\begin{proof}
     This lemma can be proved by directly applying Birkhoff Ergodic theorem to our setup.   
\end{proof}

\begin{corollary}
 Let $\Lambda$ be a self-conformal set with open set condition, and $\T$ be periodic with period $1$. The specific condition $\Omega^+{(\T,\theta,v)}=\Omega^-{(\T,\theta,v)}$ in Theorem 1 is satisfied.
\end{corollary}
\begin{proof}
    This lemma can be obtained by definition.
\end{proof}

In these cases, the Hausdorff dimension of the set of eventually always hitting points can be illustrated by Figure \ref{fig:intersection-of-curves2}. We can notice from the picture that the $s$-coordinate of the intersection is our best candidate of the Hausdorff dimension of the eventually always hitting set. It only fails to be the correct Hausdorff dimension when its $s$-coordinate is larger than the topological dimension of the domain of the IFS. In that case, the Hausdorff dimension of the eventually always hitting set is equal to the Hausdorff dimension of the attractor and the topological dimension of the domain of the IFS. 
\begin{figure}[htbp]
  \centering
  \begin{tikzpicture}
    \begin{axis}[
      width=0.5\textwidth,
      xlabel={$s$},
      ylabel={$y$},
      legend pos=north west,
      xmin=0, xmax=2,
      ymin=0, ymax=2,
      xtick=\empty, 
      ytick=\empty, 
      ylabel={},
    ]
      \addplot[blue,mark=none, domain=0:2, samples=200, name path=curveA] {2-(5/3)*x^(1/3)};

      \addplot[red,mark=none, domain=-2:2, samples=100, name path=curveB] {2*x};

      \node[blue, right, yshift=2mm] at (axis cs:0.1,1.2) {$P(s)$}; 
      \node[red, right, yshift=2mm] at (axis cs:0.4,1.55) {$h(\T,v)s$}; 
      \draw [name intersections={of=curveA and curveB, by={intersection1}}]
        [fill=black] (intersection1) circle (2pt);
     \draw[dashed] (intersection1) -- (intersection1|-{axis cs:0,0});
    \end{axis}
    \node[below, xshift=-18mm, yshift=0mm] at (current axis.south) {$\hat{s}(\T,v)$};
  \end{tikzpicture}
  \caption{Illustration of ${\dimH{\pi(R_e((a_n)_n,\T)) }}$.}
  \label{fig:intersection-of-curves2}
\end{figure}

\begin{remark}
  We can see that the bounds to the Hausdorff dimension of the eventually always hitting points in Theorem 1 are not always sharp. For example, we can take a self-similar IFS $\{f_1,f_2,f_3\}$ satisfying OSC such that $f_1$ and $f_2$ have different contraction ratio and our target is cylinders along a symbolic sequence $20011110000000000000000\cdots$ where the sequence is constructed such that after the first digit $3$, we always have $2^{2^n}$ $0$s followed by $2^{2^{n+1}}$ $1$s, etc. Our conjecture is that there always exists an increasing sequence of natural numbers $\{c_n(\T)\}_n$ dictated by the symbolic structure of the target $\T\in G$ such that $\lim_{n\to \infty}\frac{\log\|f'_{G(c_n,\T,\theta,v)}\|}{c_n}=\Omega^+{(\T,\theta,v)}=\Omega^-{(\T,\theta,v)}$ exists. If this were the case, we could find the Hausdorff dimension of the eventually always hitting points using similar proofs. However, at this point, we do not know how to construct $(c_n(\T)_n)$ that satisfies this property.  
\end{remark}

\begin{remark}
    Rigorously speaking, Theorem 1 cannot directly imply the result in \cite{bugeaud}. The reason is that, they are considering eventually always hitting set with multiple targets $\T_1,\T_2$ such that $\pi(\T_1)=\pi(\T_2)$. Generalising this conclusion is not easy, as their method relies on the equal contraction ratio of the IFS. We tried to generalize our result to multiple targets' scenario, the difficulty is that it is hard to find alternative $\Omega^+(\T,\theta,v)$ and $\Omega^-(\T,\theta,v)$ functioning well. However, the current work indeed generalize the Cantor type set results in the end of \cite{bugeaud} and the symbolical problems hidden beneath \cite{bugeaud}.
\end{remark}

\section*{Acknowledgments}
The author would like to express his gratitude to Henna Koivusalo and Sascha Troscheit for the proofreading and suggestions.
\newpage

\printbibliography 

\end{document}